\documentclass[10pt,a4paper]{article}

\usepackage[T1]{fontenc}
\usepackage{amssymb}
\usepackage{amsthm}
\usepackage{mathtools}
\usepackage{enumitem}
\usepackage{aliascnt}
\usepackage{hyperref}
\numberwithin{equation}{section}
\newtheorem{theorem}{Theorem}[section]

\newaliascnt{proposition}{theorem}

\aliascntresetthe{proposition}

\newaliascnt{lemma}{theorem}
\newtheorem{lemma}[lemma]{Lemma}
\aliascntresetthe{lemma}

\newaliascnt{corollary}{theorem}
\newtheorem{corollary}[corollary]{Corollary}
\aliascntresetthe{corollary}

\newaliascnt{conjecture}{theorem}
\newtheorem{conjecture}[conjecture]{Conjecture}
\aliascntresetthe{conjecture}

\theoremstyle{definition}
\newaliascnt{claim}{theorem}
\newtheorem{claim}[claim]{Claim}
\aliascntresetthe{claim}

\usepackage[noabbrev,capitalize]{cleveref}
\crefname{equation}{}{}

\crefname{theorem}{Theorem}{Theorems}
\crefname{proposition}{Proposition}{Propositions}
\crefname{lemma}{Lemma}{Lemmas}
\crefname{corollary}{Corollary}{Corollaries}
\crefname{conjecture}{Conjecture}{Conjectures}
\crefname{claim}{Claim}{Claims}

\newcommand{\F}{\mathcal F}
\newcommand{\G}{\mathcal G}
\newcommand{\A}{\mathcal A}
\newcommand{\B}{\mathcal B}
\newcommand{\Y}{\mathcal Y}

\begin{document}

\title{New Extremal Ranges and Constructions of the Erd\H{o}s--Kleitman Problem}
\author{
        Cheng Chi\thanks{School of Mathematical Sciences, Shanghai Jiao Tong University, 800 Dongchuan Road, Shanghai 200240, China.
                Email: chengchi@sjtu.edu.cn.
                Supported by National Key R\&D Program of China under grant No. 2022YFA1006400 and National Natural Science Foundation of China No. 12571376.
        }
        \qquad
        Yan Wang\thanks{School of Mathematical Sciences, Shanghai Jiao Tong University, 800 Dongchuan Road, Shanghai 200240, China.
                Email: yan.w@sjtu.edu.cn.
                Supported by National Key R\&D Program of China under grant No. 2022YFA1006400, National Natural Science Foundation of China under grant No. 12571376.
        }
}
\date{}
\maketitle

\begin{abstract}
        For integers $n\ge s\ge2$, let $e(n,s)$ denote the maximum size of a family $\F\subseteq2^{[n]}$ with no $s$ pairwise disjoint members.
        The problem of determining $e(n,s)$, now called the Erd\H{o}s--Kleitman problem, is the non-uniform analogue of the Erd\H{o}s matching conjecture.
        We prove that for every fixed $m\ge3$, there exist constants $\beta_m$ and $\delta_m$ such that for sufficiently large $s$, the extremal families for $e(ms+c,s)$ are
        \[
                \mathcal P'(m,s,\ell;L')\coloneqq \binom{L'}m\cup\binom{[ms+c]}{\ge m+1}
        \]
        for some $L'$ with $\ell=s-c$ and $|L'|=m\ell-1$, when $\beta_m s^{(m-1)/m}\le c\le \delta_m s$.
        This determines the extremal families in an unknown range when $\ell$ is large, complementing our earlier work on the range when $\ell$ is small.

        Moreover, for $m=3$, we sharpen this to the asymptotically optimal range.
        Let
        \[
                t(s)=
                \frac{17-18s+\sqrt{49-852s+1284s^2}}{20}
                =0.8916\cdots s+O(1)
        \]
        We prove that \(\mathcal P'(3,s,\ell;L')\) is the unique extremal family when $t(s)<\ell<s-((4/3)^{1/3}+o(1))s^{2/3}$.
        Note that the lower bound \(t(s)\) of $\ell$ is exact, while the the constant \((4/3)^{1/3}\) in the upper bound of $\ell$ is best possible.

        Kupavskii and Sokolov introduced four candidate extremal families and conjectured that the value of $e(n,s)$ is the maximum of their sizes.
        We disprove this conjecture by constructing a new family $\mathcal R(m,s,\ell)$ that is larger than each of their four proposed candidates when $\alpha_{\mathrm R}s^{1/2}\le c\le \beta_{\mathrm R}s^{(m-1)/m}$ for some constants $\alpha_{\mathrm R}$ and $\beta_{\mathrm R}$.
        This also shows that the exponent $(m-1)/m$ in the first result is tight.

\end{abstract}

\section{Introduction}
Let $ [n]\coloneqq \{1,2,\ldots,n\} $.
For a set $X$ and an integer $k$, write $\binom Xk$ for the family of all $k$-element subsets of $X$, and write $\binom X{\ge k}\coloneqq \bigcup_{i\ge k}\binom Xi$.
For an integer $n$, we also write $\binom n{\ge k}\coloneqq \sum_{i=k}^n \binom ni$.
A \emph{matching} in a family is a collection of pairwise disjoint sets.
An \emph{$s$-matching} is a matching of size $s$.
For a family $\F$, let $\nu(\F)$ denote its matching number.

For integers $n\ge s\ge2$, define
\[
        e(n,s)
        \coloneqq
        \max\{ |\F|:\F\subseteq2^{[n]}\text{ and }\nu(\F)<s\}.
\]
The problem of determining $e(n,s)$, now known as the Erd\H{o}s--Kleitman problem, goes back to Erd\H{o}s \cite{ErdosKleitman1974,ErdosKoRado}.
Kleitman determined $e(ms-1,s)$ and $e(ms,s)$ for all $m,s\ge1$ \cite{Kleitman}.
Apart from these two residue classes, the exact value of $e(n,s)$ is more sensitive and may depend on the position of $n$ between consecutive multiples of $s$.
Thus, write $n=ms+c=(m+1)s-\ell$ for some integers $c,\ell$ with $c+\ell=s$ and $0\le c<s$ from now on.

The Erd\H{o}s--Kleitman problem may be viewed as the non-uniform counterpart of the Erd\H{o}s matching conjecture (EMC, for short).
For integers $N,k,t$ with $N\ge kt$, define
\[
        e_k(N,t)
        \coloneqq
        \max\left\{|\G|:\G\subseteq\binom{[N]}k
        \text{ and } \nu(\G)<t\right\}.
\]
The Erd\H{o}s matching conjecture asserts that
\[
        e_k(N,t)
        =
        \max\left\{
        \binom Nk-\binom{N-t+1}k,\,
        \binom{kt-1}k
        \right\}.
\]
Erd\H{o}s proved the conjecture for every fixed $k$ and $t$ when $N$ is sufficiently large \cite{ErdosEMC}.
For $k=2$, EMC is known by a theorem of Erd\H{o}s and Gallai \cite{ErdosGallai}.
For $k=3$, it was proved for sufficiently large $N$ by \L{}uczak and Mieczkowska \cite{LuczakMieczkowska}, and was later completely solved by Frankl \cite{FranklDAM}.
For $k\ge4$, the conjecture remains open in general, although several important ranges and variants are known; see, for example, \cite{Frankl2017NewRange,FK2022,FranklLuMaWu2026,FranklWangOverlapping,GaoLuMaYu2022,KupavskiiRainbow}.

The connection between the Erd\H{o}s--Kleitman problem and EMC is captured by the following construction.
If $\G\subseteq\binom{[n]}m$ satisfies $\nu(\G)<\ell$, then $\G\cup\binom{[n]}{\ge m+1}$ has no $s$-matching.
Indeed, any $s$-matching using $q$ sets from $\G$ has total size at least $qm+(s-q)(m+1)=(m+1)s-q$.
Since $n=(m+1)s-\ell$, this implies $q\ge\ell$, contradicting $\nu(\G)<\ell$.
For $L'\in\binom{[n]}{m\ell-1}$, applying this construction with $\G=\binom{L'}m$ gives an admissible family
\[
        \mathcal P'(m,s,\ell;L')
        \coloneqq
        \binom{L'}m\cup\binom{[n]}{\ge m+1}.
\]
Write $\mathcal P'(m,s,\ell)$ for $\mathcal P'(m,s,\ell;[m\ell-1])$.
The size of $\mathcal P'(m,s,\ell;L')$ is independent of the choice of $L'$, and we denote it by $|\mathcal P'(m,s,\ell)|$.

However, $\mathcal P'(m,s,\ell)$ does not dominate in all ranges.
When $\ell$ is small, a different construction gives larger families that contain sets of size strictly less than $m$.
For $L\in\binom{[n]}{\ell-1}$, define
\[
        \mathcal P(m,s,\ell;L)
        \coloneqq
        \{A\subseteq[n]: |A|+|A\cap L|\ge m+1\}.
\]
Write $\mathcal P(m,s,\ell)=\mathcal P(m,s,\ell;[\ell-1])$.
Again, the size is independent of the choice of $L$.
The family $\mathcal P(m,s,\ell;L)$ has matching number strictly less than $s$.
In fact, if $A_1,\ldots,A_s$ are pairwise disjoint members of $\mathcal P(m,s,\ell;L)$, then a contradiction follows from
\begin{equation}\label{equ:P-less-than-s}
        s(m+1)-1=n+|L|
        \ge
        \sum_{i=1}^s \bigl(|A_i|+|A_i\cap L|\bigr)
        \ge s(m+1).
\end{equation}
Kupavskii and Sokolov gave two further constructions:
\[
        \begin{aligned}
                \mathcal Q(m,s,\ell)
                 & \coloneqq\{E\in 2^{[n]}\colon |E|+|E\cap[ms-c-1]|\ge2m\}, \\
                \mathcal W(m,s,\ell)
                 & \coloneqq\{E\in 2^{[n]}\colon |E\cap [ms-1]|\ge m\}.
        \end{aligned}
\]
The same counting argument as in \eqref{equ:P-less-than-s} shows that both of them have matching number less than $s$.

These four constructions motivated Kupavskii and Sokolov to propose the following conjecture.
\begin{conjecture}[Kupavskii and Sokolov \cite{KupavskiiSokolovOtherEnd}]\label{conj}
        For every $n\ge s$, writing $n=ms+c=(m+1)s-\ell$ with $0\le c<s$, we have
        \[
                e(n,s)=
                \max\{
                |\mathcal P(m,s,\ell)|,
                |\mathcal P'(m,s,\ell)|,
                |\mathcal Q(m,s,\ell)|,
                |\mathcal W(m,s,\ell)|
                \}.
        \]
\end{conjecture}
Kupavskii and Sokolov proved the conjecture for $m=2$ \cite{KupavskiiSokolov2025}, and also in a range where $c$ is sufficiently small relative to $s$ \cite{KupavskiiSokolovOtherEnd}.
In \cite{ChiWangK1}, the authors of this paper showed that, for every fixed $m\ge3$, the unique extremal family up to isomorphism is $\mathcal P(m,s,\ell)$ when $1\le \ell\le (\frac{m+1}{2m+1}-o(1))s$.
This answers another conjecture of Frankl and Kupavskii \cite{FKnonuniform} in a strong sense and also confirms \cref{conj} in this range.

This paper studies the other range of $\ell$ for the Erd\H{o}s-Kleitman problem.
Our first contribution is to prove that for $m\ge3$ and sufficiently large $s$, the family $\mathcal P'(m,s,\ell;L')$ is the unique extremal family for $e(ms+c,s)$ when $c=s-\ell$ is moderately small compared to $s$.

\begin{theorem}\label{thm:main}
        Fix $m\ge3$.
        There exist constants $\beta_m,\delta_m>0$ and an integer $s_0=s_0(m)$ such that the following holds for all $s\ge s_0$.
        Let $n=ms+c,\ell=s-c$ and $\beta_m s^{(m-1)/m}\le c\le \delta_m s$.
        If $\F\subseteq 2^{[n]}$ satisfies $\nu(\F)<s$, then $|\F|\le |\mathcal P'(m,s,\ell)|$.
        Moreover, equality holds if and only if $\F=\mathcal P'(m,s,\ell;L')$ for some $L'\in\binom{[n]}{m\ell-1}$.
\end{theorem}

The lower bound $c\ge\beta_m s^{(m-1)/m}$ in \cref{thm:main} is tight up to a multiplicative constant.
To show this, we introduce the following construction:
\[
        \mathcal R(m,s,\ell)
        \coloneqq
        \{E\in 2^{[n]}:(m-1)|E|+|E\cap[ms-(m-1)c-1]|\ge m^2\}.
\]
This family $\mathcal R(m,s,\ell)$ has matching number less than $s$.
In fact, if $A_1,\ldots,A_s$ are disjoint sets in $\mathcal R(m,s,\ell)$, then a contradiction follows from
\[
        sm^2-1=(m-1)n+\bigl(ms-(m-1)c-1\bigr)\ge\sum_{i=1}^s \bigl((m-1)|A_i|+|A_i\cap[ms-(m-1)c-1]|\bigr)\ge sm^2.
\]
For $c$ below the order $s^{(m-1)/m}$, this construction $\mathcal R(m,s,\ell)$ is larger than $\mathcal P'(m,s,\ell)$, see \cref{thm:conj}.

For $m=3$, we obtain a sharper result which identifies the asymptotically optimal range of $\mathcal P'$.
There are two natural transition points of $\ell$.
The first is the exact value of $\ell$ at which $\mathcal P(3,s,\ell)$ and $\mathcal P'(3,s,\ell)$ have the same size.
Let $t(s)$ be the root of $|\mathcal P(3,s,\ell)|=|\mathcal P'(3,s,\ell)|$ with $t(s)>1$, namely
\[
        t(s)
        \coloneqq
        \frac{17-18s+\sqrt{49-852s+1284s^2}}{20}
        =0.8916\ldots s+O(1).
\]
In \cite{ChiWangK1}, the authors of this paper showed that $\mathcal P(3,s,\ell)$ is the unique extremal family when $1\le\ell<t(s)$, while $\mathcal P(3,s,\ell)$ and $\mathcal P'(3,s,\ell)$ are both extremal when $t(s)$ is an integer and $\ell=t(s)$.
Moreover, $|\mathcal P'(3,s,\ell)|>|\mathcal P(3,s,\ell)|$ when $\ell>t(s)$.

The second transition point arises due to the family $\mathcal R(m,s,\ell)$.
Let $\alpha=(4/3)^{1/3}$, and let $r(s)$ be the root in $(0,s)$ of $|\mathcal P'(3,s,\ell)|=|\mathcal R(3,s,\ell)|$.
A direct calculation shows that $s-r(s)=\alpha s^{2/3}+o(s^{2/3})$.
Moreover, $|\mathcal P'(3,s,\ell)|>|\mathcal R(3,s,\ell)|$ for $\ell<r(s)$, whereas the reverse inequality holds for $\ell>r(s)$.

The second contribution of this paper is to determine the asymptotic range for which $\mathcal P'(3,s,\ell)$ is the unique extremal family for $e(n,s)$.

\begin{theorem}\label{thm:m3-high-range}
        For every $\varepsilon>0$, there exists an integer $s_0=s_0(\varepsilon)$ such that the following holds for all $s\ge s_0$.
        Suppose $n=4s-\ell$ with $t(s)<\ell<s-(\alpha+\varepsilon)s^{2/3}$.
        If $\F\subseteq 2^{[n]}$ satisfies $\nu(\F)<s$, then $|\F|\le|\mathcal P'(3,s,\ell)|$.
        Moreover, equality holds if and only if $\F=\mathcal P'(3,s,\ell;L')$ for some $L'\in\binom{[n]}{3\ell-1}$.
\end{theorem}

Our third contribution is to show that the construction $\mathcal R(m,s,\ell)$ is not only a barrier for extending the range for which $\mathcal P'(m,s,\ell)$ is extremal, but also a counterexample to \cref{conj}.
\begin{theorem}\label{thm:conj}
        Write $n=ms+c$.
        For every fixed $m\ge3$, there exist constants $\alpha_{\mathrm R}=\alpha_{\mathrm R}(m)>0$ and $\beta_{\mathrm R}=\beta_{\mathrm R}(m)>0$ such that \cref{conj} does not hold whenever $s$ is sufficiently large and $\alpha_{\mathrm R}s^{1/2}\le c\le \beta_{\mathrm R}s^{(m-1)/m}$.
\end{theorem}

The rest of the paper is organized as follows.
In \cref{sec:tools}, we collect the notation and auxiliary results.
In \cref{sec:comparison}, we set up the layer-by-layer comparison with the family $\mathcal P'(m,s,\ell;L')$.
In \cref{sec:general-proof}, we prove \cref{thm:main}.
In \cref{sec:m3-high}, we prove \cref{thm:m3-high-range}.
In \cref{sec:counterexample}, we prove \cref{thm:conj}.
The technical estimates used in the proof of \cref{sec:m3-high} are collected in the appendix.

\section{Definitions and lemmas}\label{sec:tools}

In this section we collect some lemmas and theorems used later.
For a family $G$, let $\tau(G)$ denote its \emph{vertex-cover number}.

For $k\ge3$, for the rest of the paper, we fix
\begin{equation}\label{eq:eta-def}
        \eta_k=\frac{1}{2k^{2k+1}} \text{ and }\delta_k=\frac{\eta_k}{k+1+2\eta_k}.
\end{equation}

\begin{theorem}[Frankl \cite{Frankl2017NewRange}]\label{thm:npemc}
        For every integer $k\ge3$, there is an integer $t_k$ such that the following holds.
        If $t\ge t_k$, $kt\le N<(k+\eta_k)t$, and $\G\subseteq\binom{[N]}k$ satisfies $\nu(\G)<t$, then $|\G|\le \binom{kt-1}{k}$.
        Moreover, if equality holds, then $\G=\binom Uk$ for some $U\in\binom{[N]}{kt-1}$.
\end{theorem}

We shall use the following two elementary estimates.
\begin{lemma}\label{lem:window}
        Let $\ell=s-c$ and $m\ge3$.
        If $c\le\delta_m s$, then
        \[
                \eta_m\ell-(m+1)c
                \ge
                \frac{\eta_m^2}{m+1+2\eta_m}s.
        \]
        In particular, $(m+1)c<\eta_m\ell$.
\end{lemma}

\begin{proof}
        By $c\le\delta_m s$ and \eqref{eq:eta-def}, we have
        \[
                \eta_m\ell-(m+1)c
                =\eta_m s-(m+1+\eta_m)c
                \ge
                \bigl(\eta_m-(m+1+\eta_m)\delta_m\bigr)s=\frac{\eta_m^2}{m+1+2\eta_m}s.
        \]
        This completes the proof.
\end{proof}

\begin{lemma}\label{lem:gap}
        Let $\ell=s-c$, $n=sm+c$ and $m\ge3$.
        Assume $c\le\delta_m s$.
        For all sufficiently large $s$, if $1\le j\le m-1$ and $p=\ell+j-m-1$, then
        \[
                \binom{m\ell-1}{m}-
                \binom{mp-1}{m}
                >
                (j+m)\binom n{m-1}.
        \]
\end{lemma}

\begin{proof}
        The choice of $\delta_m$ implies $\delta_m\le 1/(100m(m+1))$.
        Let $x=m\ell-1$ and $r=m(m+1-j)$.
        Since $mp-1=x-r$, $\binom{m\ell-1}{m}-\binom{mp-1}{m}=\sum_{a=0}^{r-1}\binom{x-a-1}{m-1}$, and this sum is at least $r\binom{x-r}{m-1}$.
        We compare $\binom{x-r}{m-1}$ with $\binom n{m-1}$.
        Since $x=m(s-c)-1$, $n=ms+c$, and $r=m(m+1-j)$, we have $n-(x-r-m+2)=(m+1)c+r+m-1$.
        Moreover, $r\le m(m+1)$, and by the choice of $\delta_m$, $(m+1)c\le s/(100m)$.
        Taking $s$ sufficiently large, we have $n-(x-r-m+2)\le s/(50m)\le n/(16m^2)$, where the last inequality uses $n\ge ms$.
        Therefore $(x-r-m+2)/n\ge 1-1/(16m^2)$.

        We have
        \[
                \frac{\binom{x-r}{m-1}}{\binom n{m-1}}
                =\prod_{i=0}^{m-2}\frac{x-r-i}{n-i}
                \ge
                \left(\frac{x-r-m+2}{n}\right)^{m-1}
                \ge
                \left(1-\frac1{16m^2}\right)^{m-1}
                \ge
                1-\frac1{16m}.
        \]
        Here the last inequality follows from Bernoulli's inequality.
        Thus $\binom{x-r}{m-1}\ge (1-1/(16m))\binom n{m-1}$.

        It suffices to show that $r(1-1/(16m))>j+m$.
        Indeed, since $r=m(m+1-j)$,
        \[
                m(m+1-j)\left(1-\frac1{16m}\right)-(j+m)
                =
                m(m+1-j)-\frac{m+1-j}{16}-(j+m)\ge \frac{7}{8},
        \]
        where the last inequality follows from taking derivative and letting $j=m-1$.
        Hence the desired strict inequality follows.
\end{proof}

In \cite{ChiWangK1}, the authors use averaging arguments to prove the following lemma.
\begin{lemma}[Chi and Wang \cite{ChiWangK1}]\label{lem:blocker}
        Let $k\ge2$, let $G\subseteq\binom Xk$, and suppose that $|X|=k\tau+\rho$, where $\tau\ge1$ and $\rho\ge0$.
        Let $Z=\binom Xk\setminus G$.
        If $\nu(G)<\tau$, then
        \[
                |Z|\ge
                \max\left\{
                \frac1{\tau}\binom{k\tau+\rho}{k},
                \binom{\rho+k}{k}
                \right\}.
        \]
\end{lemma}

Based on this lemma, we prove two much stronger lemmas for later proofs.

\begin{lemma}\label{lem:residual-blocker}
        Fix $k\ge3$.
        There are constants $\kappa_{\ref{lem:residual-blocker}}=\kappa_{\ref{lem:residual-blocker}}(k)>0$ and $\gamma_{\ref{lem:residual-blocker}}=\gamma_{\ref{lem:residual-blocker}}(k)>0$ such that the following holds.
        Let $\gamma\ge\gamma_{\ref{lem:residual-blocker}}$, let $\tau\ge1$ and $\rho\ge0$ be integers with $\gamma=\tau+\rho$, and let $W$ be a set of size $k\tau+\rho$.
        If $\A\subseteq\binom Wk$ satisfies $\nu(\A)<\tau$ and $\B=\binom Wk\setminus\A$, then $|\B|\ge \kappa_{\ref{lem:residual-blocker}}(\rho+1)\gamma^{k-1}$.
\end{lemma}

\begin{proof}
        Let $\theta=\eta_k/(2+\eta_k)$.
        Choose $\gamma_{\ref{lem:residual-blocker}}$ sufficiently large so that $(1-\theta)\gamma_{\ref{lem:residual-blocker}}\ge t_k$, where \(t_k\) is the constant from \cref{thm:npemc}.

        First suppose that $\rho\le\theta\gamma$.
        Since $\gamma=\tau+\rho$, this gives $\rho\le \theta(\tau+\rho)$, and hence $\rho\le \theta\tau/(1-\theta)={\eta_k}\tau/2$.
        Therefore
        \[
                |W|=k\tau+\rho
                \le \left(k+\frac{\eta_k}{2}\right)\tau
                <(k+\eta_k)\tau.
        \]
        Also, since $\rho\le\theta\gamma$, we have $\tau=\gamma-\rho\ge (1-\theta)\gamma$.
        Since $\rho\le\theta\gamma$, we have $\tau=\gamma-\rho\ge(1-\theta)\gamma
                \ge(1-\theta)\gamma_{\ref{lem:residual-blocker}}
                \ge t_k$.
        Together with \(k\tau\le |W|<(k+\eta_k)\tau\), this allows us to apply \cref{thm:npemc}, giving $|\A|\le\binom{k\tau-1}{k}$.
        Consequently,
        \[
                |\B|
                \ge
                \binom{k\tau+\rho}{k}-\binom{k\tau-1}{k}=
                \sum_{i=0}^{\rho}\binom{k\tau-1+i}{k-1}\ge
                (\rho+1)\binom{k\tau-1}{k-1}
        \]

        Since $\rho\le{\eta_k}\tau/2$, we have $\gamma=\tau+\rho\le \left(1+\eta_k/2\right)\tau\le (1+\eta_k)\tau$.
        Therefore, $\tau\ge \frac{\gamma}{1+\eta_k}$.
        For every $0\le j\le k-2$, since $\gamma_{\ref{lem:residual-blocker}}$ is sufficiently large,
        \[
                k\tau-1-j
                \ge
                \frac{k\gamma}{1+\eta_k}-(k-1)
                \ge
                \frac{k\gamma}{3(1+\eta_k)}.
        \]
        Thus
        \[
                \binom{k\tau-1}{k-1}
                =
                \frac{\prod_{j=0}^{k-2}(k\tau-1-j)}{(k-1)!}
                \ge
                \frac{1}{(k-1)!}
                \left(\frac{k\gamma}{3(1+\eta_k)}\right)^{k-1}.
        \]
        Hence in this case
        \[
                |\B|
                \ge
                \frac{1}{(k-1)!}
                \left(\frac{k}{3(1+\eta_k)}\right)^{k-1}
                (\rho+1)\gamma^{k-1}.
        \]

        Suppose instead that $\rho>\theta\gamma$.
        Let $M$ be a maximum matching in $\A$.
        Since $\nu(\A)<\tau$, the matching $M$ has at most $\tau-1$ edges and covers at most $k(\tau-1)$ vertices.
        Therefore at least
        \[
                |W|-k(\tau-1)
                =
                k\tau+\rho-k(\tau-1)
                =
                \rho+k
        \]
        vertices remain uncovered.
        By maximality of $M$, no $k$-set contained in the uncovered set belongs to $\A$ as otherwise such a $k$-set could be added to $M$, contradicting the maximality of $M$.
        Hence every $k$-set contained in the uncovered set belongs to $\B$, and so $|\B|\ge\binom{\rho+k}{k}$.
        Since $\rho>\theta\gamma$, we have $\rho^{k-1}>\theta^{k-1}\gamma^{k-1}$.
        Also $\rho\ge1$, because $\rho$ is an integer and $\rho>\theta\gamma>0$.
        Hence, $\rho\ge (\rho+1)/2$.
        Therefore
        \[
                \binom{\rho+k}{k}
                \ge
                \frac{\rho^k}{k!}
                \ge
                \frac{\theta^{k-1}}{2k!}(\rho+1)\gamma^{k-1}.
        \]
        Thus in this case
        \[
                |\B|
                \ge
                \frac{\theta^{k-1}}{2k!}(\rho+1)\gamma^{k-1}.
        \]
        Hence, the lemma holds by taking $\kappa_{\ref{lem:residual-blocker}}$ sufficiently small.
\end{proof}

Given a set $A$ and an integer $q$, write $A=A_1\sqcup\cdots\sqcup A_q$ if $A$ is the disjoint union of $A_1,\ldots,A_q$.
Given integers $p,u,b,h$, let $P,X$ be disjoint sets with $|P|=pu$ and
$|X|=bu+h$. For each $e\in\binom Pp$, let
$\G_e^{(b)}\subseteq\binom Xb$ and
$\G_e^{(b+1)}\subseteq\binom X{b+1}$.
A \emph{mixed ordered partition} is a pair of ordered partitions $P=e_1\sqcup\cdots\sqcup e_u$ and $X=T_1\sqcup\cdots\sqcup T_u$ such that $|e_r|=p$ for every $r$, exactly $h$ of the $T_r$'s have
size $b+1$, the remaining $u-h$ have size $b$, and
$T_r\in\G_{e_r}^{(|T_r|)}$ for every $r\in[u]$.

\begin{lemma}\label{lem:p-ordered-local-blocker}
        Let $m,p$ be integers with $m\ge2$ and $1\le p\le m-1$, and set
        $b=m+1-p$. There exist constants
        $\gamma_{\ref{lem:p-ordered-local-blocker}}>0$,
        $\rho_{\ref{lem:p-ordered-local-blocker}}>0$, and
        $C_{\ref{lem:p-ordered-local-blocker}}$ depending only on
        $m,p$ such that the following holds.

        Let $u,h$ be integers with
        $u\ge C_{\ref{lem:p-ordered-local-blocker}}$ and
        $0\le h\le\rho_{\ref{lem:p-ordered-local-blocker}}u$.
        Let $P,X$ and the families $\G_e^{(b)},\G_e^{(b+1)}$ be as in the definition above.
        If no mixed ordered partition exists, then
        \[
                \sum_{e\in\binom Pp}
                \left(
                \left|\binom Xb\setminus\G_e^{(b)}\right|
                +
                \left|\binom X{b+1}\setminus\G_e^{(b+1)}\right|
                \right)
                \ge \gamma_{\ref{lem:p-ordered-local-blocker}}(h+1)u^m.
        \]
\end{lemma}

\begin{proof}
        For each $e\in\binom Pp$, write $\mathcal B_e^{(b)}=\binom Xb\setminus\G_e^{(b)}$ and $\mathcal B_e^{(b+1)}=\binom X{b+1}\setminus\G_e^{(b+1)}$.
        Also write
        \[
                \mathcal Q_b=\{(e,T):e\in\binom Pp,\ T\in\mathcal B_e^{(b)}\} \text{ and }\mathcal Q_{b+1}=\{(e,T):e\in\binom Pp,\ T\in\mathcal B_e^{(b+1)}\}.
        \]
        We choose $\rho_{\ref{lem:p-ordered-local-blocker}}\le1/4$, and later choose $C_{\ref{lem:p-ordered-local-blocker}}$ sufficiently large.
        All constants below depend only on $m$ and $p$.

        First suppose $h=0$.
        Let $\mathcal P_0$ be the set of all ordered pairs of partitions $P=e_1\sqcup\cdots\sqcup e_u$ and $X=T_1\sqcup\cdots\sqcup T_u$, where $|e_r|=p$ and $|T_r|=b$ for every $r$.
        Since no mixed ordered partition exists, every member of $\mathcal P_0$ contains at least one forbidden pair from $\mathcal Q_b$.

        Let $I_0=\binom{pu}{p}\binom{bu}{b}$ be the number of possible $(p,b)$-pairs $(e,T)$.
        Fix such a pair $(e,T)$.
        Let $M_0$ be the number of members of $\mathcal P_0$ in which $e_r=e$ and $T_r=T$ for some $r\in[u]$.
        Then
        \[
                M_0
                =
                u\frac{(p(u-1))!}{(p!)^{u-1}}\frac{(b(u-1))!}{(b!)^{u-1}}
                =
                \frac{|\mathcal P_0|u}{I_0}.
        \]
        Indeed, one may first choose the common position $r$, then partitions $P\setminus e$ into $u-1$ ordered $p$-sets and $X\setminus T$ into $u-1$ ordered $b$-sets.

        Since every member of $\mathcal P_0$ contains at least one forbidden pair from $\mathcal Q_b$, we have $|\mathcal Q_b|M_0\ge |\mathcal P_0|$.
        Therefore,
        \[
                |\mathcal Q_b|
                \ge
                \frac{|\mathcal P_0|}{M_0}
                =
                \frac{I_0}{u}
                =
                \frac1u\binom{pu}{p}\binom{bu}{b}
                \ge c_0u^m
        \]
        for a constant $c_0=c_0(m,p)>0$, since $p+b=m+1$ and $u\ge C_{\ref{lem:p-ordered-local-blocker}}$.
        This proves the result when $h=0$.

        Now assume $h\ge1$.
        A \emph{base configuration} is a triple $\mathcal C_0=(\mathbf e,O,\mathbf T)$, where $P=e_1\sqcup\cdots\sqcup e_u$ with $|e_r|=p$, $O\in\binom Xh$, and $X\setminus O=T_1\sqcup\cdots\sqcup T_u$ with $|T_r|=b$.
        Let $\mathcal C$ be the set of all base configurations.

        For a base configuration $\mathcal C_0=(\mathbf e,O,\mathbf T)$, call $(e_r,T_r)$ the $r$-th position.
        Thus each position pairs the $r$-th $p$-block of $P$ with the $r$-th $b$-block of $X\setminus O$.
        We also write this position as $(r,e_r,T_r)$ to include its index.
        Here $e_r\subseteq P$ and $T_r\subseteq X\setminus O$.
        The overflow set $O$ is used only later, when a position is upgraded by replacing $T_r$ with $T_r\cup\{z\}$ for some $z\in O$.

        A position $(r,e_r,T_r)$ of $\mathcal C_0$ is called \emph{$b$-forbidden} if $T_r\in\mathcal B_{e_r}^{(b)}$.
        First suppose that at least $|\mathcal C|/2$ base configurations have at least $h+1$ $b$-forbidden positions.

        Let $I_b=\binom{pu}{p}\binom{bu+h}{b}$ be the number of possible $(p,b)$-pairs.
        Fix such a pair $(e,T)$. The number of base configurations in which $(e,T)$ occurs as a position $(e_r,T_r)$ is
        \[
                M_b
                =
                u\frac{(p(u-1))!}{(p!)^{u-1}}\binom{bu+h-b}{h}\frac{(b(u-1))!}{(b!)^{u-1}}
                =
                \frac{|\mathcal C|u}{I_b}.
        \]
        Indeed, one chooses the common position $r$, partitions $P\setminus e$, chooses the overflow set $O\subseteq X\setminus T$ of size $h$, and then partitions $X\setminus(O\cup T)$ into $u-1$ ordered $b$-sets.
        Hence, $|\mathcal Q_b|M_b\ge\frac12|\mathcal C|(h+1)$, and therefore
        \[
                |\mathcal Q_b|
                \ge
                \frac{|\mathcal C|(h+1)I_b}{2|\mathcal C|u}
                =
                \frac{h+1}{2}\frac{I_b}{u}
                \ge
                c_1(h+1)u^m,
        \]
        since $I_b\ge c'_1u^{m+1}$.
        This gives the desired conclusion in this case.

        We may therefore assume that more than half of the base configurations have at most $h$ $b$-forbidden positions.
        Call such configurations good.

        Fix a good base configuration $\mathcal C_0=(\mathbf e,O,\mathbf T)$.
        Define the bipartite overflow graph $\Gamma_O(\mathcal C_0)$ as follows.
        One part consists of the positions $(r,e_r,T_r)$, and the other part is $O$.
        Join a position $(r,e_r,T_r)$ to $z\in O$ if $T_r\cup\{z\}\in\G_{e_r}^{(b+1)}$.
        Here an upgrade means replacing a $b$-set $T_r$ by $T_r\cup\{z\}$, where $z\in O$, so that the tail has size $b+1$.
        The upgrade is allowed precisely when $T_r\cup\{z\}\in\G_{e_r}^{(b+1)}$.
        Thus a missing edge is exactly a forbidden upgrade.

        We claim that $\Gamma_O(\mathcal C_0)$ has at least $h$ missing edges.
        Suppose to the contrary that it has less than $h$ missing edges.
        We first show that the $b$-forbidden positions can be assigned distinct vertices of $O$ by allowed upgrades.
        In fact, if no such assignment exists, then Hall's condition would give a nonempty family $Y$ of $b$-forbidden positions with $|N(Y)|<|Y|$, where $N(Y)$ denotes the neighborhood of $Y$ in the overflow graph.
        Then all pairs in $Y\times (O\setminus N(Y))$ are missing edges, so the number of missing edges is at least
        \[
                |Y|(h-|N(Y)|)
                \ge |Y|(h-|Y|+1)
                \ge h,
        \]
        because $1\le |Y|\le h$.
        This contradicts the assumption that there are fewer than $h$ missing edges.

        Hence all $b$-forbidden positions can be matched injectively into $O$ by allowed upgrades.
        After doing this, assign the remaining overflow vertices greedily to unused positions.
        When a remaining $z\in O$ is considered, fewer than $h$ positions have already been used, and $z$ is non-adjacent to fewer than $h$ positions.
        Hence fewer than $2h$ positions are unavailable.
        Since $h\le u/4$, an unused neighbor remains.

        After all overflow vertices are assigned, upgrade exactly the positions to which they were assigned.
        Every upgraded position is allowed as a $(b+1)$-tail, and every not upgraded position is not $b$-forbidden, hence is allowed as a $b$-tail.
        This gives a mixed ordered partition, contradicting the hypothesis.
        Therefore every good base configuration has at least $h$ forbidden upgrades.

        We now count these forbidden upgrades.
        Let $I_{b+1}=\binom{pu}{p}\binom{bu+h}{b+1}$ be the number of possible $(b+1)$-pairs.
        Each base configuration has exactly $uh$ upgrade incidences.
        Fix a possible $(p,b+1)$-pair $(e,U)$. The number of upgrade incidences $(\mathcal C_0,r,z)$ realizing $(e,U)$, meaning $e_r=e$ and $T_r\cup\{z\}=U$, is
        \[
                M_{b+1}
                =
                u(b+1)\frac{(p(u-1))!}{(p!)^{u-1}}\binom{bu+h-b-1}{h-1}\frac{(b(u-1))!}{(b!)^{u-1}}
                =
                \frac{|\mathcal C|uh}{I_{b+1}}.
        \]
        Indeed, one chooses the common position $r$, chooses the overflow vertex $z\in U$, sets $T_r=U\setminus\{z\}$, partitions $P\setminus e$, chooses the remaining $h-1$ overflow vertices from $X\setminus U$, and partitions the remaining vertices of $X$ into $u-1$ ordered $b$-sets.
        Since more than half of the base configurations are good and every good configuration has at least $h$ forbidden upgrades, we have
        \[
                |\mathcal Q_{b+1}|
                M_{b+1}
                \ge
                \frac12|\mathcal C|h,
        \]
        and hence
        \[
                |\mathcal Q_{b+1}|
                \ge
                \frac{I_{b+1}}{2u}
                \ge c_2 u^{m+1}.
        \]
        Since $h\le \rho_{\ref{lem:p-ordered-local-blocker}}u\le u/4$, we have $h+1\le u$, and hence $|\mathcal Q_{b+1}| \ge c_2(h+1)u^m$.
        Taking $\gamma_{\ref{lem:p-ordered-local-blocker}}\le \min\{c_0,c_1,c_2\}$ proves
        \[
                |\mathcal Q_b|+|\mathcal Q_{b+1}|
                =
                \sum_{e\in\binom Pp}
                \left(
                \left|\binom Xb\setminus\G_e^{(b)}\right|
                +
                \left|\binom X{b+1}\setminus\G_e^{(b+1)}\right|
                \right)
                \ge
                \gamma_{\ref{lem:p-ordered-local-blocker}}(h+1)u^m.
        \]
        This completes the proof.
\end{proof}

Later, in the proof of \cref{thm:m3-high-range}, we need the following results.

\begin{lemma}[Erd\H{o}s and Gallai \cite{ErdosGallai}]\label{lem:EG}
        If $G\subseteq\binom{[n]}2$ and $\nu(G)<s$, then $|G|\le\max\left\{\binom{2s-1}{2},\binom n2-\binom{n-s+1}{2}\right\}$.
\end{lemma}

\begin{theorem}[Frankl \cite{FranklDAM}]
        \label{thm:EMC3}
        Let $N$ and $t$ be positive integers with $N\ge3t-1$.
        If $\G\subseteq\binom{[N]}3$ satisfies $\nu(\G)<t$, then
        \[
                |\G|\le
                \max\left\{\binom N3-\binom{N-t+1}3,\binom{3t-1}3\right\}.
        \]
        Moreover, if $\binom{3t-1}3>\binom N3-\binom{N-t+1}3$ and equality holds with $|\G|=\binom{3t-1}3$, then $\G=\binom U3$ for some $U\in\binom{[N]}{3t-1}$.
\end{theorem}

\begin{theorem}[Frankl \cite{FranklShifting1987}]
        \label{thm:frankl-shadow-matching}
        Let $k\ge2$ and $\G\subseteq\binom Vk$ satisfy $\nu(\G)<t$.
        Then $|\G|\le(t-1)\binom {|V|}{k-1}$.
\end{theorem}

For integers $N$ and $u$, define
\[
        h_3(N,u)=\binom N3-\binom{N-u}3+1-\binom{N-u-3}{2}.
\]

\begin{theorem}[Guo, Lu and Mao \cite{GuoLuMao}]\label{thm:GLM}
        There exists an integer $N_0$ such that the following holds for all integers $N\ge N_0$ and $u\ge1$ with $N\ge3u+2$.
        If $G\subseteq\binom{[N]}3$ satisfies $\nu(G)\le u<\tau(G)$, then $|G|\le
                \max\left\{h_3(N,u),\binom{3u+2}{3}\right\}$.
\end{theorem}

For $Q\ge 0$, families $\G_1,\ldots,\G_p\subseteq 2^X$
are called \emph{$Q$-dependent} if there are no pairwise
disjoint sets $G_i\in \G_i$ such that $|G_1\cup\cdots\cup G_p|\le Q$.
\begin{theorem}[Frankl and Kupavskii \cite{FKstability}]
        \label{thm:FK-q-dependent}
        Let $p\ge2$, $q\ge0$, and $\lambda\in[p]$ be integers, and write $Q=pq+p-\lambda$.
        Suppose that $N\ge Q$ and that $\G_1,\ldots,\G_p\subseteq2^{[N]}$ are $Q$-dependent.
        Then
        \[
                \sum_{i=1}^p |\G_i|
                \le
                (\lambda-1)\binom Nq
                +p\sum_{j=q+1}^N\binom Nj .
        \]
\end{theorem}

The following result is a direct corollary of \cref{thm:FK-q-dependent}.
\begin{corollary}
        \label{cor:m3-nonuniform-terminal-blocker}
        Let $u\ge2$, $0\le h<u$, and let $W$ be a set of size $4u+h$.
        Let $\G\subseteq\binom W{\ge4}$.
        If $\nu(\G)<u$, then
        \[
                \left|\binom W{\ge4}\setminus\G\right|
                \ge
                \frac{h+1}{u}\binom{4u+h}{4}.
        \]
\end{corollary}
\begin{proof}
        Apply \cref{thm:FK-q-dependent} with $p=u$, $q=4$, and $\lambda=u-h$.
        Since $\nu(\G)<u$, there are no pairwise disjoint
        $G_1,\ldots,G_u\in\G$.
        Hence the $u$ identical families
        $\G_1=\cdots=\G_u=\G$ are $Q$-dependent for every $Q$,
        in particular for $Q=4u+h$.
        Moreover,
        \[
                pq+p-\lambda=4u+u-(u-h)=4u+h=|W|.
        \]
        Hence
        \[
                u|\G|
                \le
                (u-h-1)\binom{|W|}{4}
                +u\sum_{j=5}^{|W|}\binom{|W|}{j}.
        \]
        Subtracting this from $u\binom{|W|}{\ge4}=u\binom{|W|}{4}+u\sum_{j=5}^{|W|}\binom{|W|}{j}$ gives the desired result.
\end{proof}

\section{Comparison setup}\label{sec:comparison}

In this section we set up the comparison with
$\mathcal P'(m,s,\ell)$.
Fix $m,s,c,\ell,n$ with $m\ge3$, $n=ms+c$, and $\ell=s-c$ from now on.
For a family $\F\subseteq2^{[n]}$ and $k\ge0$, write $\F_i=\F\cap\binom{[n]}i$, $\F_{<m}=\bigcup_{i=0}^{m-1}\F_i$, $H=\F_m$, $\F_{\ge k}=\F\cap\binom{[n]}{\ge k}$, $\Y_i=\binom{[n]}i\setminus\F_i$ and $\Y_{\ge k}=\binom{[n]}{\ge k}\setminus\F_{\ge k}$.

We first use the following simple normalization to remove the empty set.
\begin{lemma}\label{lem:empty-reduction}
        Let $n\ge s\ge2$, and let $\F\subseteq2^{[n]}$ satisfy $\nu(\F)<s$.
        If $\emptyset\in\F$, then there is a family $\F'\subseteq2^{[n]}$ such that $|\F'|=|\F|$, $\nu(\F')<s$, $\emptyset\notin\F'$, and $\F'_1\ne\emptyset$.
\end{lemma}

\begin{proof}
        Since $n\ge s$ and $\nu(\F)<s$, not all singletons can belong to $\F$.
        Choose $x\in[n]$ with $\{x\}\notin\F$, and set $\F'=(\F\setminus\{\emptyset\})\cup\{\{x\}\}$.
        $|\F'|=|\F|$, $\emptyset\notin\F'$, and $\F'_1\ne\emptyset$ by construction.
        If an $s$-matching in $\F'$ avoids $\{x\}$, then it is already an $s$-matching in $\F$.
        If it uses $\{x\}$, replacing $\{x\}$ by $\emptyset$ gives an $s$-matching in $\F$.
        Thus $\nu(\F')<s$.
\end{proof}

The basic comparison with $\mathcal P'(m,s,\ell)$ is the following equation:
\begin{equation}\label{eq:layer-balance}
        |\F|-|\mathcal P'(m,s,\ell)|
        =
        |\F_{<m}|+|H|-\binom{m\ell-1}{m}
        -|\Y_{\ge m+1}|.
\end{equation}
Indeed, $|\F|
        =
        |\F_{<m}|+|H|+
        \sum_{i=m+1}^{n}\binom ni
        -|\Y_{\ge m+1}|$, while $|\mathcal P'(m,s,\ell)|
        =
        \binom{m\ell-1}{m}+
        \sum_{i=m+1}^{n}\binom ni$.
Thus, to prove $|\F|\le |\mathcal P'(m,s,\ell)|$, it is enough to show
\[
        |\F_{<m}|+|H|-\binom{m\ell-1}{m}
        \le |\Y_{\ge m+1}|.
\]

\section{Proof of \cref{thm:main}}
\label{sec:general-proof}

Fix $m\ge3$, and let $\delta_m$ be defined by \eqref{eq:eta-def}.
Let $\kappa=\kappa_{\ref{lem:residual-blocker}}(m+1)$ be the constant from \cref{lem:residual-blocker} with $k=m+1$.
Choose $\beta_m$ sufficiently large such that, for all $s\ge1$ and all $n\le(m+1)s$,
\begin{equation}\label{eq:K-choice}
        2m\binom n{m-1}\le \beta_m^m s^{m-1}
        \text{ and }
        4m\binom n{m-1}\le \kappa \beta_m^m s^{m-1}.
\end{equation}
Choose $s_0=s_0(m)$ sufficiently large so that all applications of \cref{thm:npemc,lem:residual-blocker,lem:gap} below are valid and all estimates requiring $s$ to be large hold.

Let $s\ge s_0$, let $\beta_m s^{(m-1)/m}\le c\le\delta_m s$, write $\ell=s-c$ and $n=ms+c$ and let $\F\subseteq2^{[n]}$ satisfy $\nu(\F)<s$.
We use the notation from \cref{sec:comparison}.
By \cref{lem:empty-reduction}, we may and do assume that $\emptyset\notin\F$.
By \eqref{eq:layer-balance}, it is enough to prove $|\F_{<m}|+|H|-\binom{m\ell-1}{m}
        \le
        |\Y_{\ge m+1}|$.
We shall prove the stronger estimate
\begin{equation}\label{eq:target}
        |\F_{<m}|+|H|-\binom{m\ell-1}{m}
        \le
        |\Y_{m+1}|.
\end{equation}
Also, as $\emptyset\notin\F$,
\begin{equation}\label{eq:low-layer-bound}
        |\F_{<m}|
        =
        \sum_{i=1}^{m-1}|\F_i|
        \le
        \sum_{i=1}^{m-1}\binom ni
        \le
        m\binom n{m-1}.
\end{equation}

The behavior of the $m$-layer is decisive.
We split according to whether $H$ has an $\ell$-matching.

\subsection{Case 1: $\nu(H)<\ell$}

In this case, we have $\nu(H)<\ell$.

Since \(s\ge s_0\), we have \(\ell\ge t_m\), where $t_m$ is the constant from \cref{thm:npemc}.
Also, since $n=m\ell+(m+1)c$, \cref{lem:window} gives $m\ell\le n<(m+\eta_m)\ell$.
Thus the hypotheses of \cref{thm:npemc} are satisfied.
Applying \cref{thm:npemc} gives
\begin{equation}\label{eq:H-case1-bound}
        |H|\le\binom{m\ell-1}{m}.
\end{equation}
If $|\F_{<m}|=0$, then \eqref{eq:target} follows from \eqref{eq:H-case1-bound}.
Assume $|\F_{<m}|>0$, and choose $E\in\F_j$ for some $1\le j\le m-1$.

Let $p=\ell+j-m-1$.
The value of $p$ is chosen so that one lower-layer set, $p$ disjoint $m$-sets, and $c+m-j$ disjoint $(m+1)$-sets would together form an $s$-matching.
For sufficiently large $s$, we have $p\ge1$,
\begin{equation}\label{eq:p-identities}
        1+p+(c+m-j)=s,
        \text{ and }
        n-j-mp=(m+1)(c+m-j).
\end{equation}
Let $H(\overline E)=\{Q\in H:Q\cap E=\emptyset\}$.
There are two subcases.

\medskip
\noindent{\bf Subcase 1a:} $H(\overline E)$ contains a $p$-matching.
\medskip

In this case, choose pairwise disjoint sets $Q_1,\ldots,Q_p\in H(\overline E)$.
Then $E,Q_1,\ldots,Q_p$ are pairwise disjoint members of $\F$.
Let $W=[n]\setminus(E\cup Q_1\cup\cdots\cup Q_p)$ and $\tau=c+m-j$.
By \eqref{eq:p-identities}, $|W|=n-j-mp=(m+1)(c+m-j)=(m+1)\tau$ and $1+p+\tau=1+(\ell+j-m-1)+(c+m-j)=\ell+c=s$.
If $\F_{m+1}\cap\binom W{m+1}$ contained $\tau$ pairwise disjoint sets, then these sets together with $E,Q_1,\ldots,Q_p$ would form an $s$-matching in $\F$, a contradiction.
Hence $\F_{m+1}\cap\binom W{m+1}$ has no $\tau$-matching.

Since $|W|=(m+1)\tau$, applying \cref{lem:blocker} gives $|\Y_{m+1}|\ge\left|\binom W{m+1}\setminus \F_{m+1}\right|\ge \tau^{-1}\binom{(m+1)\tau}{m+1}\ge\tau^m\ge c^m$, where the last inequality follows from $\tau=c+m-j\ge c$.
Since $c\ge \beta_m s^{(m-1)/m}$, \eqref{eq:K-choice} and \eqref{eq:low-layer-bound} give $|\F_{<m}|\le m\binom n{m-1}\le \beta_m^m s^{m-1}\le c^m\le |\Y_{m+1}|$.
Together with \eqref{eq:H-case1-bound}, this proves \eqref{eq:target} in this subcase.

\medskip
\noindent{\bf Subcase 1b: }$H(\overline E)$ has no $p$-matching.
\medskip

The family $H(\overline E)$ is a subfamily of $\binom{[n]\setminus E}{m}$ and $|[n]\setminus E|=N=n-j=mp+(m+1)(c+m-j)$.
Our goal here is to bound $|H(\overline E)|$ by \cref{thm:npemc}.
It remains to verify $N<(m+\eta_m)p$, equivalently $(m+1)(c+m-j)<\eta_m p$.
Using $p=\ell+j-m-1$, we have
\[
        \begin{aligned}
                \eta_m p-(m+1)(c+m-j)
                 & = \eta_m\ell-(m+1)c -\eta_m(m+1-j) -(m+1)(m-j)                \\
                 & \ge  \frac{\eta_m^2}{m+1+2\eta_m}s -\eta_m(m+1-j) -(m+1)(m-j) \\
                 & =\frac{\eta_m^2}{m+1+2\eta_m}s-O(1),
        \end{aligned}
\]
where the inequality follows from \cref{lem:window}.
Since $s_0$ is sufficiently large, we have $mp\le N<(m+\eta_m)p$.
\cref{thm:npemc} gives $|H(\overline E)|\le\binom{mp-1}{m}$.
The number of $m$-sets meeting $E$ is at most $\binom nm-\binom{n-j}m\le j\binom n{m-1}$.
Therefore $|H|\le j\binom n{m-1}+\binom{mp-1}{m}$.
Combining this with \eqref{eq:low-layer-bound} and \cref{lem:gap} gives $|\F_{<m}|+|H|\le(m+j)\binom n{m-1}+\binom{mp-1}{m}<\binom{m\ell-1}{m}$.
This proves \eqref{eq:target} in this subcase.

\subsection{Case 2: $\nu(H)\ge\ell$}

Write $\nu(H)=\ell+d$.
Since $\F$ has no $s$-matching and $s=\ell+c$, we must have $d\le c-1$.
Thus $0\le d\le c-1$.
The family $H$ has no $(\ell+d+1)$-matching.
Since $d\le c-1$, we have $m(\ell+d+1)\le n$.
Also
\[
        n-m(\ell+d+1)=(m+1)c-m(d+1)
        \le(m+1)c
        <\eta_m\ell
        \le\eta_m(\ell+d+1),
\]
by \cref{lem:window}.
Thus $m(\ell+d+1)\le n<(m+\eta_m)(\ell+d+1)$, and \cref{thm:npemc} gives $|H|\le\binom{m(\ell+d+1)-1}{m}$.
Consequently,
\[
        \begin{aligned}
                |H|-\binom{m\ell-1}{m}\le
                \sum_{a=0}^{m(d+1)-1}\binom{m\ell+a-1}{m-1}\le
                m(d+1)\binom n{m-1}.
        \end{aligned}
\]
Together with \eqref{eq:low-layer-bound}, this gives
\begin{equation}\label{eq:case2-excess}
        |\F_{<m}|+|H|-\binom{m\ell-1}{m}
        \le
        2m(d+1)\binom n{m-1}.
\end{equation}
Choose $\ell+d$ pairwise disjoint members $Q_1,\ldots,Q_{\ell+d}$ of $H$.
Let $W=[n]\setminus(Q_1\cup\cdots\cup Q_{\ell+d})$ and $\tau=c-d$.
Then $\tau\ge1$ and $|W|=(m+1)\tau+d$.
If $\F_{m+1}\cap\binom W{m+1}$ contained $\tau$ pairwise disjoint sets, then these sets together with $Q_1,\ldots,Q_{\ell+d}$ would form an $s$-matching in $\F$.
Hence, $\nu(\F_{m+1}\cap\binom W{m+1})<\tau$.
Applying \cref{lem:residual-blocker} to $\F_{m+1}\cap\binom W{m+1}$ with $k=m+1$, $\rho=d$, and $\gamma=\tau+d=c$, we obtain $|\Y_{m+1}|\ge\kappa(d+1)c^m$.
Since $c\ge \beta_m s^{(m-1)/m}$, \eqref{eq:K-choice} gives
\[
        |\Y_{m+1}|
        \ge
        \kappa(d+1)\beta_m^m s^{m-1}
        \ge
        4m(d+1)\binom n{m-1}.
\]
Together with \eqref{eq:case2-excess}, this proves \eqref{eq:target} in Case~2.

Therefore $|\F|\le |\mathcal P'(m,s,\ell)|$ for every family $\F$ with $\nu(\F)<s$.

\subsection{Equality cases}
First suppose that $\emptyset\notin\F$ and $|\F|=|\mathcal P'(m,s,\ell)|$.
Since $|\F|=|\mathcal P'(m,s,\ell)|$, \eqref{eq:layer-balance} gives
\begin{equation}\label{eq:equality-balance}
        |\Y_{\ge m+1}|
        =
        |\F_{<m}|+|H|-\binom{m\ell-1}{m}.
\end{equation}
Equality cannot occur in Case~2 since we have
\[
        |\Y_{m+1}|
        \ge
        4m(d+1)\binom n{m-1}
        >
        2m(d+1)\binom n{m-1}
        \ge
        |\F_{<m}|+|H|-\binom{m\ell-1}{m}
\]
in Case~2.
Since $|\Y_{\ge m+1}|\ge |\Y_{m+1}|$, this contradicts \eqref{eq:equality-balance}.

Thus equality must occur in Case~1.
If $|\F_{<m}|>0$, take $E\in\F_j$ with $1\le j\le m-1$, as in the proof of Case~1.
In Subcase~1a, we obtain
\[
        |\Y_{m+1}|
        \ge c^m
        \ge \beta_m^m s^{m-1}
        \ge 2m\binom n{m-1}
        > |\F_{<m}|.
\]
Together with $|H|\le\binom{m\ell-1}{m}$, this again contradicts \eqref{eq:equality-balance}.
In Subcase~1b, we obtain $|\F_{<m}|+|H|<\binom{m\ell-1}{m}$, making the right-hand side of \eqref{eq:equality-balance} negative, while $|\Y_{\ge m+1}|\ge0$.
Hence equality implies $|\F_{<m}|=0$.

With $|\F_{<m}|=0$, equality is possible only in Case~1 and hence $\nu(H)<\ell$.
By \cref{thm:npemc}, $|H|\le\binom{m\ell-1}{m}$.
Equation \eqref{eq:equality-balance} becomes
\[
        |\Y_{\ge m+1}|=|H|-\binom{m\ell-1}{m}.
\]
The left-hand side is non-negative, whereas the right-hand side is non-positive.
Therefore both sides are zero.
Hence $|H|=\binom{m\ell-1}{m}$ and $\Y_i=\emptyset$ for every $i\ge m+1$.
\cref{thm:npemc} gives a set $L'\in\binom{[n]}{m\ell-1}$ such that $H=\binom{L'}m$.
This implies that $\F=\mathcal P'(m,s,\ell;L')$.

It remains to exclude extremal families containing $\emptyset$.
Suppose that such a family exists.
By \cref{lem:empty-reduction}, one can replace $\emptyset$ by a missing singleton and obtain an extremal family $\F'$ with $\emptyset\notin\F'$ and $\F'_1\ne\emptyset$.
The equality classification just proved forces $\F'=\mathcal P'(m,s,\ell;L')$ for some $L'$, but $\mathcal P'(m,s,\ell;L')$ has no singleton since $m\ge3$, a contradiction.
This completes the proof.

\medskip

\noindent {\bf Remark.} We make no attempt to optimize the constants $\beta_m$ and $\delta_m$.
The value of $\delta_m$ mainly depends on the result in \cref{thm:npemc}.
More precisely, we need to use the assumption $kt\le N<(k+\eta_k)t$ in \cref{thm:npemc} with $\eta_k=1/(2k^{2k+1})$.
This forces us to ensure $n=m\ell+(m+1)c < (m+\eta_m)\ell$ holds, thus $c<\eta_m\ell/(m+1)$.
Thus we take
\[
        \delta_m=\frac{\eta_m}{m+1+2\eta_m}
        =
        \frac{1}{2(m+1)m^{2m+1}+2}.
\]
The constant $\beta_m$ is only chosen to dominate the number of sets with size at most $m-1$ and some extra losses come from \cref{lem:residual-blocker}.
From the proof of Lemma~\ref{lem:residual-blocker}, one may take
\[
        \kappa_{\ref{lem:residual-blocker}}(m+1)
        =
        \frac{1}{2(m+1)!\bigl(4(m+1)^{2m+3}+1\bigr)^m}
        \text{ and }
        \beta_m
        =
        (8m^2)^{1/m}(m+1)\bigl(4(m+1)^{2m+3}+1\bigr).
\]

\section{Proof of \cref{thm:m3-high-range}}\label{sec:m3-high}

Fix $\varepsilon>0$ and let $\alpha=(4/3)^{1/3}$.
Recall that
\[
        t(s)=\frac{17-18s+\sqrt{49-852s+1284s^2}}{20}= 0.8916\cdots s+O(1).
\]
Let $s$ be sufficiently large and let $\ell$ satisfy $t(s)<\ell<s-(\alpha+\varepsilon)s^{2/3}$.
Let $n=3s+c=4s-\ell,\ell=s-c,a=\ell-1$ and $r_0=n-a=2s+2c+1$.
Then $s-t(s)>c>(\alpha+\varepsilon)s^{2/3}$.
Also $c<0.11s$ for all sufficiently large $s$.

Throughout this section, take $m=3$ in \cref{sec:comparison}.
For each family $\F\subseteq2^{[n]}$, write $\Y_{\ge4}=\binom{[n]}{\ge4}\setminus\F_{\ge4}=\bigcup_{i=4}^n\Y_i$.
For a family $\G$ and a set $E$, write $\G(\overline E)=\{G\in\G:G\cap E=\emptyset\}$.
Define $A_3=\binom{3\ell-1}{3}=\binom{3s-3c-1}{3}$, and $\Lambda_3=\binom a2+\binom n3-\binom{r_0}{3}$ and $B_{\le2}=(s-1)+(s-1)(2n-s)/2$.
The range $c<s-t(s)$, equivalently $\ell>t(s)$, gives
\begin{equation}\label{eq:A3-beats-Lambda3}
        A_3>\Lambda_3.
\end{equation}
Indeed, direct expansion gives $\Lambda_3-A_3=
        \frac{\ell-1}{3}
        \left(24s^2-6s-6-(18s-17)\ell-10\ell^2\right)$.
The polynomial $Q_s(x)=10x^2+(18s-17)x-24s^2+6s+6$ has positive leading coefficient and $t(s)$ is its positive root.
Since $\ell>t(s)$, we have $Q_s(\ell)>0$. As $\ell>1$, it follows that $\Lambda_3-A_3<0$, proving \eqref{eq:A3-beats-Lambda3}.

For $i\in\{1,2\}$, let
\begin{equation}\label{eq:m3-pi-Ni-Li}
        p_i=\ell+i-4, \text{ }
        N_i=n-i \text{ and }
        L_i=\binom n3-\binom{n-i}3.
\end{equation}
Thus $L_i$ is the number of triples meeting a fixed $i$-set.
We shall use the following shifted estimates.
\begin{claim}\label{clm:m3-shifted-estimates}
        For every $i\in\{1,2\}$,
        \begin{align}
                A_3-L_i-h_3(N_i,p_i-1)    & >B_{\le2}, \label{eq:m3-shifted-h3}     \\
                A_3-L_i-\binom{3p_i-1}{3} & >B_{\le2}. \label{eq:m3-shifted-clique}
        \end{align}
\end{claim}
The proof of \cref{clm:m3-shifted-estimates} is given in the appendix for readability.

The proof is organized according to the structure of the $3$-layer $H=\F_3$.
If $H$ has a small vertex cover, then many triples are
missing from a large set.
If $\nu(H)\le \ell-1<\tau(H)$, we use \cref{thm:GLM}.
If $H$ already contains at least $\ell$ disjoint triples, then the remaining vertices must miss many
sets of size at least $4$.

\subsection{Reduction to the key inequality}

Let $\F\subseteq2^{[n]}$ satisfy $\nu(\F)<s$.
Since the family $\mathcal P'(3,s,\ell;L')$ contains all sets of size at least $4$, the only possible surplus over $\mathcal P'(3,s,\ell;L')$ comes from layers $1,2,3$, and this surplus must be paid for by missing sets above the $3$-layer.

By \cref{lem:empty-reduction}, we may and do assume that $\emptyset\notin\F$ from now on.
Since $n=3s+c$ and $c<0.11s$, we have $2n-s=5s+2c<5.22s$ for all sufficiently large $s$.
Hence
\begin{equation}\label{eq:Ble2-less-3s2}
        B_{\le2}<3s^2.
\end{equation}
Since $\emptyset\notin\F$, \eqref{eq:layer-balance} gives
\[
        |\F|-|\mathcal P'(3,s,\ell)|
        =
        |\F_1|+|\F_2|+|H|-A_3-|\Y_{\ge4}|.
\]
Thus it is enough to prove
\begin{equation}\label{eq:key-high-terminal}
        |\F_1|+|\F_2|+|H|\le A_3+|\Y_{\ge4}|.
\end{equation}

We first bound the number of $1$-sets and $2$-sets in $\F$.
The family $\F_1$ has no $s$ pairwise disjoint members, so $|\F_1|\le s-1$.
Also $\nu(\F_2)<s$.
By \cref{lem:EG}, and since $n=3s+c$ with $c>0$,
\[
        |\F_2|\le \max\left\{\binom{2s-1}{2},\binom n2-\binom{n-s+1}{2}\right\}
        =\binom n2-\binom{n-s+1}{2}=\frac{(s-1)(2n-s)}{2}.
\]
Thus, by the definition of $B_{\le2}$,
\begin{equation}\label{eq:F12-bound-terminal}
        |\F_1|+|\F_2|\le B_{\le2}.
\end{equation}

\subsection{Case 1: $\tau(H)\le a$}

In this case, let $A$ be a vertex cover of $H$ with $|A|=a$.
Let $R_0=[n]\setminus A$ and $Z_3=\binom{[n]}3\setminus H$.
Since $A$ covers $H$, every triple contained in $R_0$ is missing from $H$, so $\binom{R_0}{3}\subseteq Z_3$.

Define
\[
        B_1(Z_3)=\left\{x\in[n]: |Z_3(\overline{\{x\}})|\ge\binom{r_0+2}{3}\right\} \text{ and }B_2(Z_3)=\left\{E\in\binom{[n]}2: |Z_3(\overline E)|\ge\binom{r_0}{3}\right\}.
\]
For a singleton $E=\{x\}\in\F_1$, call $E$ \emph{bad} if $x\in B_1(Z_3)$.
For $E\in\F_2$, call $E$ bad if $E\in B_2(Z_3)$.
All other sets in $\F_1\cup\F_2$ are called \emph{good}.

Let $Z_3^+=Z_3\setminus\binom{R_0}{3}$ and $\xi=|Z_3^+|=|Z_3|-\binom{r_0}{3}$.
Every pair contained in $A$ belongs to $B_2(Z_3)$, because $\binom{R_0}{3}\subseteq Z_3$; these pairs contribute exactly $\binom a2$.
We claim that
\begin{equation}\label{eq:terminal-core-bound}
        |B_1(Z_3)|+|B_2(Z_3)|\le\binom a2+|Z_3|-\binom{r_0}{3}.
\end{equation}
It remains to bound the bad singletons and the bad pairs not wholly contained in $A$ by $\xi$.
Split them into the four classes
\[
        \begin{aligned}
                \mathcal U_1 & =B_2(Z_3)\cap\{E:|E|=2,
                |E\cap A|=1\},\text{ }
                \mathcal U_2=B_2(Z_3)\cap\binom{R_0}{2}, \\
                \mathcal U_3 & =B_1(Z_3)\cap A
                \text{ and }
                \mathcal U_4=B_1(Z_3)\cap R_0.
        \end{aligned}
\]
Let
\[
        D_1=\binom{r_0-1}{2},\text{ }
        D_2=(r_0-2)^2,\text{ }
        D_3=r_0^2\text{ and }
        D_4=\frac{3r_0^2-3r_0+2}{2}.
\]
Now we show that if $\mathcal U_i$ is non-empty for some $i$, then $\xi\ge D_i$.
In fact, if $E\in\mathcal U_1$, then $|E\cap A|=|E\cap R_0|=1$.
Exactly $\binom{r_0-1}{3}$ triples from the core $\binom{R_0}{3}$ are disjoint from $E$, while the definition of $B_2(Z_3)$ requires at least $\binom{r_0}{3}$ missing triples disjoint from $E$.
Hence,
\[
        \xi
        \ge |Z_3^+(\overline E)|
        \ge \binom{r_0}{3}-\binom{r_0-1}{3}
        =D_1.
\]
If $E\in\mathcal U_2$, then $E\in\binom{R_0}{2}$.
The family $\binom{R_0}{3}$ contributes only the $\binom{r_0-2}{3}$ triples contained in $R_0\setminus E$, but the definition of $B_2(Z_3)$ again requires $\binom{r_0}{3}$ missing triples outside $E$.
Therefore,
\[
        \xi
        \ge |Z_3^+(\overline E)|
        \ge \binom{r_0}{3}-\binom{r_0-2}{3}
        =D_2.
\]
If $x\in\mathcal U_3$, then $x\in A$.
All triples in $\binom{R_0}{3}$ are disjoint from $x$, $\binom{R_0}{3}$ already accounts for $\binom{r_0}{3}$ missing triples.
Since $x\in B_1(Z_3)$, the threshold is $\binom{r_0+2}{3}$, and hence
\[
        \xi
        \ge |Z_3^+(\overline{\{x\}})|
        \ge \binom{r_0+2}{3}-\binom{r_0}{3}
        =D_3.
\]
Finally, if $x\in\mathcal U_4$, then $x\in R_0$.
The family $\binom{R_0}{3}$ contributes $\binom{r_0-1}{3}$  triples disjoint from $x$, whereas the threshold for $B_1(Z_3)$ is still $\binom{r_0+2}{3}$.
Thus
\[
        \xi
        \ge |Z_3^+(\overline{\{x\}})|
        \ge \binom{r_0+2}{3}-\binom{r_0-1}{3}
        =D_4.
\]
Consequently, every member of $\mathcal U_j$ forces $\xi\ge D_j$.
Now it remains to prove \eqref{eq:terminal-core-bound}.

In fact, for $j=1,2,3,4$, the number of possible elements in $\mathcal U_1\cup\cdots\cup \mathcal U_j$ is at most $N^{\prime}_j$, where
\[
        N^{\prime}_1=ar_0,\text{ }
        N^{\prime}_2=ar_0+\binom{r_0}{2},\text{ }
        N^{\prime}_3=ar_0+\binom{r_0}{2}+a\text{ and }
        N^{\prime}_4=ar_0+\binom{r_0}{2}+a+r_0.
\]
Since $r_0-2a=4c+3$, we have $r_0\ge2a+7$ for all sufficiently large $s$.
These thresholds are increasing, since
\[
        D_2-D_1=\frac{(r_0-3)(r_0-2)}2,\text{ }
        D_3-D_2=4(r_0-1)\text{ and }
        D_4-D_3=\frac{(r_0-2)(r_0-1)}2.
\]
For fixed $a$, each difference $D_j-N^{\prime}_j$ increases with $r_0$.
At $r_0=2a+7$, the four differences are
\[
        4a+15,\text{ } 4,\text{ }7a+28\text{ and }2a^2+16a+36.
\]
Hence $N^{\prime}_j\le D_j-1$ for every $j$.
If $\xi<D_1$, all four classes are empty.
If $D_j\le\xi<D_{j+1}$ for some $j<4$, only the first $j$ classes can be non-empty, so their union has size at most $N^{\prime}_j\le D_j-1\le\xi$.
If $\xi\ge D_4$, the whole union has size at most $N^{\prime}_4\le D_4-1\le\xi$.
This proves \eqref{eq:terminal-core-bound}.

\medskip
\noindent {\bf Subcase 1a: there are no good sets.}
\medskip

In this subcase, we have $\F_1\subseteq B_1(Z_3)$ and $\F_2\subseteq B_2(Z_3)$.
By \eqref{eq:terminal-core-bound}, $|\F_1|+|\F_2|\le\binom a2+|Z_3|-\binom{r_0}{3}$.
Since $|H|=\binom n3-|Z_3|$, it follows that
\[
        |\F_1|+|\F_2|+|H|\le\binom a2+\binom n3-\binom{r_0}{3}=\Lambda_3<A_3,
\]
where the strict inequality follows from \eqref{eq:A3-beats-Lambda3}.
Thus \eqref{eq:key-high-terminal} holds strictly.

\medskip
\noindent {\bf Subcase 1b: there exists a good set.}
\medskip

In this subcase, let $E\in\F_i$ be a good set with $i\in\{1,2\}$.
Let $p=p_i$ and $N=N_i$.
If $i=1$, then $E=\{x\}$ with $x\notin B_1(Z_3)$; if $i=2$, then $E\notin B_2(Z_3)$.
In both cases, $|Z_3(\overline E)|<\binom{r_0+4-2i}{3}$.
The identity $r_0+4-2i=(n-i)-p+1=N-p+1$ gives
\begin{equation}\label{eq:good-lower-HbarE}
        |H(\overline E)|>\binom N3-\binom{N-p+1}{3}.
\end{equation}

If $H(\overline E)$ contains $p$ pairwise disjoint triples, choose such triples $Q_1,\ldots,Q_p$.
Let $W=[n]\setminus(E\cup Q_1\cup\cdots\cup Q_p)$ and $u=s-(p+1)=c+3-i$.
Then $|W|=4u$.
If $\F_4\cap\binom W4$ contained $u$ pairwise disjoint $4$-sets, these sets together with $E,Q_1,\ldots,Q_p$ would form an $s$-matching in $\F$.
Hence $\nu(\F_4\cap\binom W4)<u$, and \cref{lem:blocker} gives
\[
        |\Y_4|\ge \left|\binom W4\setminus\F_4\right|
        \ge \frac1u\binom{4u}{4}.
\]
Since $u\ge c+1>(\alpha+\varepsilon)s^{2/3}$ and $(\alpha+\varepsilon)^3>4/3$,
\[
        \frac1u\binom{4u}{4}
        =\frac{(4u-1)(4u-2)(4u-3)}6
        =\frac{32}{3}u^3+O(u^2)
        \ge \frac{16}{3}u^3
        \ge \frac{16}{3}(\alpha+\varepsilon)^3s^2
        >4s^2
\]
for all sufficiently large $s$.
By \eqref{eq:Ble2-less-3s2}, $|\Y_4|>B_{\le2}$.
Also $|H|\le\binom n3-\binom{r_0}{3}<A_3$, because $A$ covers $H$ and $A_3>\Lambda_3$.
Using \eqref{eq:F12-bound-terminal}, we obtain $|\F_1|+|\F_2|+|H|<A_3+|\Y_4|\le A_3+|\Y_{\ge4}|$.

It remains to handle the subcase in which $H(\overline E)$ contains no $p$ pairwise disjoint triples.
Note that $N-(3p-1)=4(s-\ell)+13-4i>0$, so $N\ge3p$ and in particular $N\ge3p-1$.
Since $H(\overline E)$ has no $p$ pairwise disjoint $3$-sets, \cref{thm:EMC3} gives
\[
        |H(\overline E)|\le
        \max\left\{\binom N3-\binom{N-p+1}{3},\binom{3p-1}{3}\right\}=\binom{3p-1}{3},
\]
where the equality follows from \eqref{eq:good-lower-HbarE}.
By the definition of $L_i$, at most $L_i$ triples of $H$ meet $E$, hence $|H|\le L_i+\binom{3p_i-1}{3}$.
By \eqref{eq:F12-bound-terminal} and \eqref{eq:m3-shifted-clique}, we have $|\F_1|+|\F_2|+|H|<A_3$.
This completes Case 1.

\subsection{Case 2: $\nu(H)\le a<\tau(H)$}

In this case, we have $n-(3a+2)=4s-\ell-(3\ell-1)=4(s-\ell)+1>0$.
Since $A_3=\binom{3\ell-1}{3}=\binom{3a+2}{3}$, \cref{thm:GLM} gives
\[
        |H|\le\max\left\{h_3(n,a),\binom{3a+2}{3}\right\}=\max\{h_3(n,a),A_3\}.
\]
Now we consider two subcases.

\medskip
\noindent {\bf Subcase 2a:} $|H|\le h_3(n,a)$.
\medskip

By \eqref{eq:m3-shifted-h3}, since $L_2+h_3(N_2,p_2-1)=h_3(n,a)$ by \eqref{eq:m3-pi-Ni-Li} and $a=\ell-1$, we have $A_3-h_3(n,a)>B_{\le2}$.
Together with \eqref{eq:F12-bound-terminal}, we obtain
\[
        |\F_1|+|\F_2|+|H|<B_{\le2}+h_3(n,a)<A_3.
\]

Hence \eqref{eq:key-high-terminal} holds strictly.

\medskip
\noindent {\bf Subcase 2b:} $h_3(n,a)<|H|\le A_3$.
\medskip

In this subcase, if $\F_1\cup\F_2=\emptyset$, then $|\F_1|+|\F_2|+|H|=|H|\le A_3$, so \eqref{eq:key-high-terminal} holds.
Suppose therefore that $E\in\F_i$ for some $i\in\{1,2\}$, and write $p=p_i$ and $N=N_i$.

If $H(\overline E)$ contains $p$ pairwise disjoint triples, choose such triples $Q_1,\ldots,Q_p$.
Let
\[
        W=[n]\setminus\left(E\cup Q_1\cup\cdots\cup Q_p\right)
        \text{ and }
        u=s-p-1=c+3-i.
\]
Then $|W|=4u$.
If $\F_4\cap\binom W4$ contained $u$ pairwise disjoint $4$-sets, then these sets together with $E,Q_1,\ldots,Q_p$ would form an $s$-matching in $\F$, a contradiction.
Hence $\nu(\F_4\cap\binom W4)<u$.
Applying \cref{lem:blocker} gives
\[
        |\Y_4|\ge\left|\binom W4\setminus\F_4\right|
        \ge \frac1u\binom{4u}{4}.
\]
Since $u\ge c+1>(\alpha+\varepsilon)s^{2/3}$ and $(\alpha+\varepsilon)^3>4/3$, we have, for all sufficiently large $s$,
\[
        \frac1u\binom{4u}{4}
        =\frac{(4u-1)(4u-2)(4u-3)}6
        =\frac{32}{3}u^3+O(u^2)
        \ge \frac{16}{3}u^3
        \ge \frac{16}{3}(\alpha+\varepsilon)^3s^2
        >4s^2>B_{\le2},
\]
where the last inequality uses \eqref{eq:Ble2-less-3s2}.
Since $|H|\le A_3$, \eqref{eq:F12-bound-terminal} gives $|\F_1|+|\F_2|+|H|<A_3+|\Y_4|\le A_3+|\Y_{\ge4}|$.

It remains to consider the case $\nu(H(\overline E))\le p-1$.
We claim that $\tau(H(\overline E))>p-1$.
Otherwise, a vertex cover of $H(\overline E)$ of size at most $p-1$, together with the vertices of $E$, would cover every triple in $H$.
Its size would be at most $(p-1)+i=\ell+2i-5\le\ell-1=a$, contradicting $\tau(H)>a$.

Thus $\nu(H(\overline E))\le p-1<\tau(H(\overline E))$.
To apply \cref{thm:GLM} with $N=N_i$ and $u=p_i-1$, we need to verify that $N_i\ge 3p_i-1$, which follows from
\[
        N_i-(3(p_i-1)+2)=N_i-(3p_i-1)=4(s-\ell)+13-4i>0.
\]
Thus \cref{thm:GLM} gives $|H(\overline E)|\le\max\{h_3(N,p-1),\binom{3p-1}{3}\}$.
Hence
\[
        |\F_1|+|\F_2|+|H|
        \le
        B_{\le2}+L_i+\max\left\{h_3(N_i,p_i-1),\binom{3p_i-1}{3}\right\}
        <A_3,
\]
where the first inequality uses \eqref{eq:F12-bound-terminal} and the fact that at most $L_i$ triples of $H$ meet $E$, and the second follows from \eqref{eq:m3-shifted-h3} and \eqref{eq:m3-shifted-clique}.
This completes Case 2.

\subsection{Case 3: $\nu(H)\ge\ell$}

Write $\nu(H)=\ell+d$ for some $d\ge0$.
Since $\nu(\F)<s$ and $s=\ell+c$, we have $0\le d\le c-1$.
For $0\le d\le c-1$, define
\begin{equation}\label{eq:Sd-exact-expansion}
        S_d\coloneqq\binom{3\ell+3d+2}{3}-\binom{3\ell-1}{3}
        =\frac{d+1}{2}
        \left(27\ell^2+27d\ell+9d^2-9\ell+2\right).
\end{equation}

Since $H$ has no $(\ell+d+1)$-matching and $n-(3(\ell+d+1)-1)=4c-3d-2\ge c+1>0$, \cref{thm:EMC3} gives
\begin{equation}\label{equ:H-maximum}
        |H|\le
        \max\left\{
        \binom n3-\binom{n-\ell-d}{3},
        \binom{3\ell+3d+2}{3}
        \right\}.
\end{equation}
Recall that $A_3=\binom{3\ell-1}{3}$.
We now show that the maximum is attained by the second term, and hence bound the possible surplus of $H$ over $A_3$.
\begin{claim}\label{clm:second-term}
        For every $0\le d\le c-1$, $\binom{3\ell+3d+2}{3}>\binom n3-\binom{n-\ell-d}{3}$.
        Consequently, $|H|\le \binom{3\ell+3d+2}{3}$ and $|H|-A_3\le S_d$.
\end{claim}
The proof of \cref{clm:second-term} is a polynomial verification and is given in the appendix for readability.

In the rest of the proof, our goal is to show that
\begin{equation}\label{eq:m3-case3-direct-target}
        |\Y_{\ge4}|
        \ge
        |H|-A_3+|\F_1|+|\F_2|+1.
\end{equation}
Indeed, \eqref{eq:m3-case3-direct-target} implies $|\F_1|+|\F_2|+|H| \le A_3+|\Y_{\ge4}|-1 <A_3+|\Y_{\ge4}|$, which is \eqref{eq:key-high-terminal} strictly.
In several subcases we prove the stronger estimate
\begin{equation}\label{eq:m3-case3-target}
        |\Y_{\ge4}|
        \ge
        S_d+|\F_1|+|\F_2|+1,
\end{equation}
which is sufficient by \cref{clm:second-term}.

\medskip

Let $T_0=(\alpha+\varepsilon)^3$ and $\beta_0=27/(2T_0)$.
Since $\alpha^3=4/3$, we have $T_0>4/3$, and hence $\beta_0<81/8<32/3$.
By continuity at $x=0$, choose constants $0<\rho_0<1/4$ and $\kappa_0>0$, depending only on $\varepsilon$, such that $\frac{(4-3x)^4}{24(1-x)}
        \ge \beta_0+3\kappa_0$ for all $0\le x\le\rho_0$.
The remaining parameters used to split the range $d>\rho_0 c$ will be chosen after Subcase~3a.

Now we split the proof based on the value of $d$.
The three ranges of $d$ use different sources of loss above the $3$-layer.
For small $d$, the residual vertex set is close to a multiple of $4$, so \cref{cor:m3-nonuniform-terminal-blocker} suffices.
For intermediate $d$, either many $4$-sets can first be extracted,
allowing \cref{lem:p-ordered-local-blocker}, or there are few $4$-sets and the missing $4$-sets already give the required loss by \cref{thm:frankl-shadow-matching}.
For large $d$, we apply \cref{lem:p-ordered-local-blocker} to find missing sets to pay for the surplus in $3$-layer.

\medskip
\noindent\textbf{Subcase 3a: $0\le d\le\rho_0 c$.}
\medskip

In this case, let $\mathcal N\subseteq H$ be a matching of size $\ell+d$, let $u=c-d$ and write $W=[n]\setminus V(\mathcal N)$.
Then $|W|=4u+d=4c-3d$.
Since $d\le \rho_0 c$ and $\rho_0<1/4$, we have $0\le d<u$.
Since $\mathcal N$ already contains $\ell+d$ members of $\F$ and $\nu(\mathcal F)<s$, the family $\F_{\ge4}\cap\binom W{\ge4}$ has no $u$-matching.
Applying \cref{cor:m3-nonuniform-terminal-blocker} with $h=d$ gives
\begin{equation}\label{eq:small-basic-high-loss}
        |\Y_{\ge4}|
        \ge
        \frac{d+1}{c-d}\binom{4c-3d}{4}.
\end{equation}
Let $x=d/c$.
Using $\binom z4\ge (z^4-6z^3)/24$ and the choice of $\rho_0$ and $\kappa_0$, we obtain, uniformly for $0\le x\le\rho_0$,
\[
        \frac{d+1}{c-d}\binom{4c-3d}{4}
        \ge
        (d+1)c^3
        \left(
        \frac{(4-3x)^4}{24(1-x)}
        -\frac{(4-3x)^3}{4(1-x)c}
        \right)
        \ge (d+1)c^3(\beta_0+2\kappa_0)
\]
for all sufficiently large $s$.
On the other hand, by \eqref{eq:Sd-exact-expansion}, $\ell\le s$ and $c\ge(\alpha+\varepsilon)s^{2/3}$,
\[
        \frac{S_d}{(d+1)c^3}
        \le
        \frac{27}{2}\frac{s^2}{c^3}
        +\frac{27}{2}\frac{ds}{c^3}
        +\frac{9}{2}\frac{d^2}{c^3}
        +\frac1{c^3}.
\]
The first term is at most $\beta_0$, and the remaining terms are $o(1)$ uniformly for $d\le\rho_0 c$.
Since $s_0$ is sufficiently large, we have
\[
        S_d\le (d+1)c^3(\beta_0+\kappa_0).
\]
Combining the last two estimates gives
\begin{equation}\label{eq:small-basic-gap}
        \frac{d+1}{c-d}\binom{4c-3d}{4}-S_d
        \ge
        \kappa_0(d+1)c^3
\end{equation}
for every $0\le d\le\rho_0 c$.

Choose an integer $A_0=A_0(\varepsilon)$ such that $\kappa_0(A_0+1)(\alpha+\varepsilon)^3>4$.
By \eqref{eq:small-basic-high-loss} and \eqref{eq:small-basic-gap}, we have, throughout Subcase~3a,
\[
        |\Y_{\ge4}|-S_d \ge \kappa_0(d+1)c^3.
\]
This already implies \eqref{eq:m3-case3-target} unless $d<A_0$ and $\F_1\cup\F_2\ne\emptyset$.
Indeed, by \eqref{eq:Ble2-less-3s2}, $|\F_1|+|\F_2|<3s^2$.
If $d\ge A_0$, then
\[
        |\F_1|+|\F_2|+1<4s^2<\kappa_0(A_0+1)(\alpha+\varepsilon)^3s^2\le \kappa_0(d+1)c^3
\]
for all sufficiently large $s$.
In fact, if $0\le d<A_0$ and $\F_1\cup\F_2=\emptyset$, then $|\F_1|+|\F_2|=0$, while $\kappa_0(d+1)c^3\ge1$ for all sufficiently large $s$.
Thus \eqref{eq:m3-case3-target} holds in both cases.

It remains to consider $0\le d<A_0$ and $\F_1\cup\F_2\ne\emptyset$.
Choose $E\in\F_j$ with $j\in\{1,2\}$.
First suppose that $H(\overline E)$ contains an $(\ell+d)$-matching $\mathcal N_E$.
Let $W_E=[n]\setminus\bigl(E\cup V(\mathcal N_E)\bigr)$.
Then
\[
        |W_E|=4c-3d-j=4(c-d-1)+(d+4-j).
\]
If $\F_{\ge4}\cap\binom {W_E}{\ge4}$ contained a $(c-d-1)$-matching, then this matching, together with $E$ and $\mathcal N_E$, would give an $s$-matching in $\F$.
Thus $\nu\left(\F_{\ge4}\cap\binom {W_E}{\ge4}\right)<c-d-1$.
Since $d<A_0$, for all sufficiently large $s$, we have $0\le d+4-j<c-d-1$.
Applying \cref{cor:m3-nonuniform-terminal-blocker} with $u=c-d-1$ and $h=d+4-j$ gives
\[
        |\Y_{\ge4}|
        \ge
        \frac{d+5-j}{c-d-1}\binom{4c-3d-j}{4}.
\]
For each fixed pair $0\le d<A_0$ and $j\in\{1,2\}$, as $s\to\infty$,
\begin{equation}\label{eq:small-low-hasmatching-Y-asymp}
        \frac{d+5-j}{c-d-1}\binom{4c-3d-j}{4}
        =
        \left(\frac{32}{3}(d+5-j)+o(1)\right)c^3,
\end{equation}
and, by \eqref{eq:Sd-exact-expansion}, \eqref{eq:Ble2-less-3s2}, and $c\ge(\alpha+\varepsilon)s^{2/3}$,
\begin{equation}\label{eq:small-low-hasmatching-cost-asymp}
        S_d+|\F_1|+|\F_2|+1
        \le
        \left(
        \frac{\frac{27}{2}(d+1)+3}{T_0}
        +o(1)
        \right)c^3.
\end{equation}
We compare the leading $c^3$-coefficients in \eqref{eq:small-low-hasmatching-Y-asymp} and \eqref{eq:small-low-hasmatching-cost-asymp}.
For every $d\ge0$ and $j\in\{1,2\}$, their difference is positive.
Indeed, since $T_0>4/3$ and $j\le2$,
\[
        \frac{32}{3}(d+5-j)
        -\frac{\frac{27}{2}(d+1)+3}{T_0}>
        \frac{32}{3}(d+3)
        -\frac34\left(\frac{27}{2}(d+1)+3\right)=\frac{13}{24}d+\frac{157}{8}>0.
\]
Since only finitely many pairs $(d,j)$ occur, the positive coefficient gap is uniform in this range. Since $s_0$ is sufficiently large,
\[
        \frac{d+5-j}{c-d-1}\binom{4c-3d-j}{4}
        -
        \bigl(S_d+|\F_1|+|\F_2|+1\bigr)
        >0.
\]
Therefore \eqref{eq:m3-case3-target} follows.

We may therefore assume that $H(\overline E)$ has no $(\ell+d)$-matching.
In this situation, our goal is to prove \eqref{eq:m3-case3-direct-target}.
Since $d<A_0$ and $j\le2$, $n-j-(3\ell+3d-1)=4c-3d+1-j>0$ for all sufficiently large $s$.
Hence \cref{thm:EMC3} gives
\[
        |H(\overline E)|
        \le
        \max\left\{
        \binom{n-j}{3}-\binom{2s+2c-d-j+1}{3},
        \binom{3\ell+3d-1}{3}
        \right\}.
\]
The following estimate shows that the maximum is attained by the second term.
\begin{claim}\label{clm:small-shifted-second-term}
        For every $0\le d<A_0$ and every $j\in\{1,2\}$,
        \[
                \binom{3\ell+3d-1}{3}
                >
                \binom{n-j}{3}-\binom{2s+2c-d-j+1}{3}.
        \]
\end{claim}
The proof of \cref{clm:small-shifted-second-term} is a polynomial verification and is given in the appendix for readability.
By \cref{clm:small-shifted-second-term}, the maximum is attained by the second term.
Thus, $|H(\overline E)|\le \binom{3\ell+3d-1}{3}$.
By the definition of $L_j$, the number of $3$-sets meeting $E$ is at most $L_j$.
Therefore
\[
        |H|
        \le
        \binom{3\ell+3d-1}{3}+L_j.
\]
Since $c<0.11s$, for all sufficiently large $s$, $L_1=\binom{n-1}{2}<5s^2$ and $L_2=(n-2)^2<10s^2$.
Moreover, for $0\le d<A_0$, we have $\binom{3\ell+3d-1}{3}-\binom{3\ell-1}{3} \le \left(\frac{27}{2}d+o(1)\right)s^2$.
Thus this gives
\begin{equation}\label{eq:small-direct-H-surplus}
        |H|-A_3
        \le
        \left(\frac{27}{2}d+10+o(1)\right)s^2.
\end{equation}
On the other hand, \eqref{eq:small-basic-high-loss} gives, uniformly for
\(0\le d<A_0\), $|\Y_{\ge4}|
        \ge
        \left(\frac{32}{3}(d+1)+o(1)\right)c^3$.
By \eqref{eq:Ble2-less-3s2} and \eqref{eq:small-direct-H-surplus},
\[
        |H|-A_3+|\F_1|+|\F_2|+1
        \le
        \left(\frac{27}{2}d+13+o(1)\right)s^2
        \le
        \left(
        \frac{\frac{27}{2}d+13}{T_0}+o(1)
        \right)c^3,
\]
because \(c^3\ge T_0s^2\). Since \(T_0>4/3\),
\[
        \frac{32}{3}(d+1)
        -
        \frac{\frac{27}{2}d+13}{T_0}
        >
        \frac{32}{3}(d+1)
        -
        \frac34\left(\frac{27}{2}d+13\right)
        =
        \frac{13}{24}d+\frac{11}{12}>0.
\]
Since \(0\le d<A_0\) is a finite range, taking \(s_0\) sufficiently large gives $|\Y_{\ge4}|>
        |H|-A_3+|\F_1|+|\F_2|+1$, which is \eqref{eq:m3-case3-direct-target}.
This completes the proof of Subcase 3a.

\medskip

We now choose the parameters needed for the remaining range $d>\rho_0 c$.
Let $\gamma_1=\gamma_{\ref{lem:p-ordered-local-blocker}}(3,1)$, $\rho_1=\rho_{\ref{lem:p-ordered-local-blocker}}(3,1)$ and $C_1=C_{\ref{lem:p-ordered-local-blocker}}(3,1)$ be the constants from \cref{lem:p-ordered-local-blocker} with $(m,p)=(3,1)$.
For $\lambda>0$, let
\[
        g_\lambda(x)=\frac{(9-8\lambda)(4-3x)^3}{24}.
\]
Since
\[
        \beta_0
        <\frac{27}{2\alpha^3}
        =\frac{81}{8}=g_0(1/3),
\]
we may choose $0<\lambda<1/8$ sufficiently small so that, with $x_0=1/(3-2\lambda)$, one has $g_\lambda(x_0)>\beta_0$, and so that $\lambda\le \frac{3\rho_1}{1+3\rho_1}$.
Finally choose $\zeta$ sufficiently small so that
\[
        0<\zeta<\min\{\rho_1,1/2\}, \text{ }
        x_1\coloneqq\frac{1+3\zeta}{3(1+\zeta)}<x_0
        \text{ and }x_1>\rho_0 .
\]
This is possible because $x_0>1/3$, $x_1\to1/3$ as $\zeta\to0$, and $\rho_0<1/4$.
We split the remaining range $d>\rho_0 c$ into $\rho_0 c<d\le x_1c$ and $x_1c<d\le c-1$.
We shall also use the following crude bound. From \eqref{eq:Sd-exact-expansion}, $d\le c\le s$, and $d+1\le c$, we have
\begin{equation}\label{eq:Sd-crude-cs2}
        S_d\le 32cs^2
\end{equation}
for all sufficiently large $s$, uniformly for $0\le d\le c-1$.

\medskip
\noindent\textbf{Subcase 3b: $\rho_0 c< d\le x_1c$.}
\medskip

Let $\mathcal N=\{E_1,\ldots,E_{\ell+d}\}\subseteq H$ be a matching of size $\ell+d$, and write $W=[n]\setminus V(\mathcal N)$.
Then $|W|=4c-3d$.
Let $r_1=\lfloor\lambda d\rfloor$, $u=3(d-r_1)$ and $q=c-3d+2r_1$.
First note that $q$ is positive. Since $d\le x_1c<x_0c$ and $r_1=\lfloor\lambda d\rfloor$, we have
\[
        q=c-3d+2r_1
        \ge \bigl(1-(3-2\lambda)x_1\bigr)c-O(1).
\]
By the definition of $x_0$, the coefficient of $c$ is positive. Thus $q>0$, since $s$ is sufficiently large.

Let $\G_4=\F_4\cap\binom W4$.
Now we consider two branches depending on whether $\nu(\G_4)\ge q$.

\medskip
\noindent{\bf First branch: $\nu(\G_4)\ge q$.}
\medskip

Fix a $q$-matching $\mathcal Q\subseteq\G_4$ and let $X=W\setminus V(\mathcal Q)$.
Then
\[
        |X|=(4c-3d)-4(c-3d+2r_1)=9d-8r_1=3u+r_1.
\]
For every $I\in\binom{[\ell+d]}{d-r_1}$, write $P_I=\bigcup_{i\in I}E_i$.
Since $\mathcal N$ is a matching and $|I|=d-r_1$, the set $P_I$ has size $3(d-r_1)=u$.
The matching edges $E_i$ with $i\notin I$ form an $(\ell+r_1)$-matching.
For $x\in P_I$, define
\[
        \G_x^{(3)}=\{T\in\binom X3:x\cup T\in\F_4\} \text{ and }
        \G_x^{(4)}=\{T\in\binom X4:x\cup T\in\F_5\}.
\]
If a mixed ordered partition\footnote{See the statement of \cref{lem:p-ordered-local-blocker}.} of $P_I$ and $X$ existed, with exactly $r_1$ tails of size $4$, then it would give $u$ pairwise disjoint members of $\F_4\cup\F_5$, all contained in $P_I\cup X$.
These $u$ sets are disjoint from the $q$-matching $\mathcal Q$, because $X=W\setminus V(\mathcal Q)$, and they are also disjoint from the triples $E_i$ with $i\notin I$, because their only vertices in $V(\mathcal N)$ lie in $P_I\subseteq\bigcup_{i\in I}E_i$.
Thus they combine with the $(\ell+r_1)$-matching $\{E_i:i\notin I\}$ and the $q$-matching $\mathcal Q$ to give
\[
        (\ell+r_1)+q+u
        =\ell+r_1+(c-3d+2r_1)+3(d-r_1)=s
\]
pairwise disjoint members of $\F$, a contradiction.

Here $|P_I|=u$ and $|X|=3u+r_1$.
Moreover, $r_1\le\lambda d$ and $u=3(d-r_1)\ge3(1-\lambda)d$; by the choice of $\lambda$, this gives $r_1\le\rho_1u$.
Finally, since $d\ge\rho_0 c$ and $c\to\infty$, we have $u\ge C_1$ for all sufficiently large $s$.
Hence \cref{lem:p-ordered-local-blocker}, applied with $(m,p,b,h)=(3,1,3,r_1)$, gives
\[
        \sum_{x\in P_I}
        \left(
        \left|\binom X3\setminus\G_x^{(3)}\right|
        +
        \left|\binom X4\setminus\G_x^{(4)}\right|
        \right)
        \ge
        \gamma_1(r_1+1)u^3.
\]
We sum this estimate over all $I\in\binom{[\ell+d]}{d-r_1}$.
A counted missing set has the form $x\cup T$, where $x\in V(\mathcal N)$ and $T\subseteq X\subseteq W$.
Since $W\cap V(\mathcal N)=\emptyset$, such a set contains exactly one vertex of $V(\mathcal N)$, namely $x$.
If $x\in E_i$, then $I$ must contain $i$.
Therefore a fixed missing set is counted for at most $\binom{\ell+d-1}{d-r_1-1}$ choices of $I$.
Thus, using $c<0.11s$, $d\le c$, $d\ge\rho_0 c$, and $r_1=\lfloor\lambda d\rfloor$, we get
\begin{equation}\label{eq:middle-ob-Y-lower}
        \begin{aligned}
                |\Y_{\ge4}| & \ge
                \frac{\binom{\ell+d}{d-r_1}}{\binom{\ell+d-1}{d-r_1-1}}
                \gamma_1(r_1+1)u^3
                =
                \frac{\ell+d}{d-r_1}\gamma_1(r_1+1)u^3 \\
                            & \ge
                \frac{s}{2c}\gamma_1(\lambda\rho_0 c)
                \bigl(3(1-\lambda)\rho_0 c\bigr)^3
                =
                \frac12\gamma_1\lambda\rho_0
                \bigl(3(1-\lambda)\rho_0\bigr)^3sc^3.
        \end{aligned}
\end{equation}
Together with \eqref{eq:Sd-crude-cs2} and \eqref{eq:Ble2-less-3s2}, we have
\[
        S_d+|\F_1|+|\F_2|+1\le 32cs^2+3s^2+1=o(sc^3),
\]
since $c\ge(\alpha+\varepsilon)s^{2/3}$.
Therefore, \eqref{eq:middle-ob-Y-lower} implies \eqref{eq:m3-case3-target} for all sufficiently large $s$.

\medskip
\noindent{\bf Second branch: $\nu(\G_4)<q$.}
\medskip

In this branch, let $x=d/c$ and denote $w=|W|=(4-3x)c$.
Then $x\in[\rho_0,x_1]$.
Since $r_1=\lambda d+O(1)$, we have
\[
        q-1=\bigl(1-(3-2\lambda)x\bigr)c+O(1),
\]
uniformly for $\rho_0\le x\le x_1$.
By \cref{thm:frankl-shadow-matching}, we obtain $|\G_4|\le(q-1)\binom w3$.
Consequently,
\begin{align}
        |\Y_4\cap\binom W4|
         & \ge
        \binom w4-(q-1)\binom w3                                                     \notag   \\
         & =
        \frac{(4-3x)^4}{24}c^4
        -\frac{(1-(3-2\lambda)x)(4-3x)^3}{6}c^4
        +O(c^3)                                                                        \notag \\
         & =
        xg_\lambda(x)c^4+O(c^3),
        \label{eq:middle-shadow-loss}
\end{align}
where the $O(c^3)$ term is uniform for $\rho_0\le x\le x_1$.
Set
\[
        \eta_1\coloneqq\frac12\min_{\rho_0\le y\le x_1}
        y\bigl(g_\lambda(y)-\beta_0\bigr)>0.
\]
This is positive since $g_\lambda$ is decreasing, $x_1<x_0$, and $g_\lambda(x_0)>\beta_0$.
By \eqref{eq:middle-shadow-loss}, since $s_0$ is sufficiently large,
\[
        |\Y_4\cap\binom W4|
        \ge
        \left(\beta_0 x+\eta_1\right)c^4
        =
        \left(\beta_0\frac dc+\eta_1\right)c^4.
\]
On the other hand, using \eqref{eq:Sd-exact-expansion} and $0\le d\le c$, we have $S_d=\frac{27}{2}d\ell^2+R_d$ with $|R_d|\le C(s^2+c^2s+c^3)$, where $C$ is an absolute constant.
Since $c\ge(\alpha+\varepsilon)s^{2/3}$, $s^2+c^2s+c^3=o(c^4)$.
Moreover,
\[
        \frac{27}{2}d\ell^2
        \le
        \frac{27}{2}ds^2
        \le
        \frac{27}{2(\alpha+\varepsilon)^3}\frac dc\,c^4
        =\beta_0\frac dc\,c^4.
\]
Hence, $S_d\le(\beta_0d/c+o(1))c^4$.
Also $|\F_1|+|\F_2|+1\le B_{\le2}+1=O(s^2)=o(c^4)$.
Combining the last two estimates, we obtain
\[
        |\Y_{\ge4}|-\bigl(S_d+|\F_1|+|\F_2|+1\bigr)
        \ge
        \eta_1c^4-o(c^4)>0
\]
for all sufficiently large $s$.
Thus \eqref{eq:m3-case3-target} holds in the remaining part of Subcase 3b.
This completes the proof of Subcase 3b.

\medskip
\noindent\textbf{Subcase 3c: $x_1c< d\le c-1$.}
\medskip

Let $\mathcal N=\{E_1,\ldots,E_{\ell+d}\}\subseteq H$ be a matching of size $\ell+d$, and write $W=[n]\setminus V(\mathcal N)$, so $|W|=4c-3d$. Let
\[
        \theta=\frac{3d-(1-\zeta)c}{3+\zeta},\text{ }
        r_2=\lfloor\theta\rfloor,\text{ }
        u=c-r_2\text{ and }
        h=c-3d+3r_2 .
\]
The point of this choice is that $W$ has size $3u+h$, with $h$ a small positive multiple of $u$, while $d-r_2$ selected triples of $\mathcal N$ contain enough vertices to supply the $u$ labels so that \cref{lem:p-ordered-local-blocker} applies.

We first check that the conditions in \cref{lem:p-ordered-local-blocker} are satisfied.

\begin{claim}\label{clm:large-parameter-check}
        For all sufficiently large $s$, we have $ 0\le r_2<d$, $|W|=3u+h$, $\frac{\zeta}{2}u\le h\le\zeta u$, $u\le3(d-r_2)$ and $\frac{c}{3+\zeta}\le u\le c$.
\end{claim}
\begin{proof}
        By the definition of $x_1$ and the assumption $d\ge x_1c$, we have $\theta\ge \zeta c/(1+\zeta)$, so $r_2\ge0$ for all sufficiently large $s$. Also $d-\theta=(\zeta d+(1-\zeta)c)/(3+\zeta)>0$, and therefore $r_2<d$. The identity $|W|=3u+h$ follows from $4c-3d=3(c-r_2)+(c-3d+3r_2)$.

        The real value $\theta$ was chosen so that $c-3d+3\theta=\zeta(c-\theta)$. Since $0\le\theta-r_2<1$ and $u=c-r_2$, we have $h-\zeta u=(3+\zeta)(r_2-\theta)$, and hence $-(3+\zeta)<h-\zeta u\le0$. Thus $h\le\zeta u$, and since $u\to\infty$, we also have $h\ge(\zeta/2)u$ for all sufficiently large $s$.

        The inequality $u\le3(d-r_2)$ is equivalent to $r_2\le(3d-c)/2$. Since $r_2\le\theta$, it is enough to check $\theta\le(3d-c)/2$, which is equivalent to $d\ge x_1c$ and follows from the present assumption. Finally, $u\le c$ follows from $r_2\ge0$, while $u=c-r_2\ge c-\theta=(4c-3d)/(3+\zeta)\ge c/(3+\zeta)$ because $d\le c$.
\end{proof}

By \cref{clm:large-parameter-check}, we have $0\le h\le\rho_1u$ and $u\ge C_1$ for all sufficiently large $s$, since $\zeta<\rho_1$. Fix $I\in\binom{[\ell+d]}{d-r_2}$. Since $u\le3(d-r_2)$, choose $P\in\binom{\bigcup_{i\in I}E_i}{u}$. The matching edges $E_i$ with $i\notin I$ form an $(\ell+r_2)$-matching. For each $x\in P$, define
\[
        \G_x^{(3)}=\{T\in\binom W3:x\cup T\in\F_4\} \text{ and }
        \G_x^{(4)}=\{T\in\binom W4:x\cup T\in\F_5\}.
\]
We apply \cref{lem:p-ordered-local-blocker} with the label set $P$, the tail set $X=W$, and $(m,p,b)=(3,1,3)$. If a mixed ordered partition of $P$ and $W$ existed, with exactly $h$ tails of size $4$, then the tails of size $3$ would give members of $\F_4$, and the tails of size $4$ would give members of $\F_5$. Thus it would produce $u$ pairwise disjoint members of $\F_4\cup\F_5$, all contained in $P\cup W$. These sets are disjoint from every $E_i$ with $i\notin I$, because their vertices in $V(\mathcal N)$ lie in $P\subseteq\bigcup_{i\in I}E_i$, while their remaining vertices lie in $W$. They would therefore combine with $\{E_i:i\notin I\}$ to form $(\ell+r_2)+u=s$ pairwise disjoint members of $\F$, a contradiction. Hence
\[
        \sum_{x\in P}
        \left(
        \left|\binom W3\setminus\G_x^{(3)}\right|
        +
        \left|\binom W4\setminus\G_x^{(4)}\right|
        \right)
        \ge
        \gamma_1(h+1)u^3.
\]

We sum this estimate over all admissible pairs $(I,P)$. A missing set counted in the sum has the form $x\cup T$, where $x\in V(\mathcal N)$ and $T\subseteq W$; since $W\cap V(\mathcal N)=\emptyset$, the vertex $x$ is uniquely determined. If $x\in E_i$, then $i\in I$, so $I$ has at most $\binom{\ell+d-1}{d-r_2-1}$ possible choices. For each fixed $I$, the set $P$ must contain $x$, so $P$ has at most $\binom{3(d-r_2)-1}{u-1}$ possible choices. Since the number of admissible pairs $(I,P)$ is $\binom{\ell+d}{d-r_2}\binom{3(d-r_2)}u$, division by this maximum multiplicity gives
\[
        |\Y_{\ge4}|
        \ge
        \frac{\binom{\ell+d}{d-r_2}\binom{3(d-r_2)}u}
        {\binom{\ell+d-1}{d-r_2-1}\binom{3(d-r_2)-1}{u-1}}
        \gamma_1(h+1)u^3
        =
        3\gamma_1(\ell+d)(h+1)u^2.
\]
By \cref{clm:large-parameter-check}, $h\ge(\zeta/2)u$ and $u\ge c/(3+\zeta)$; also $\ell+d\ge\ell=s-c>s/2$ for all sufficiently large $s$. Hence
\begin{equation}\label{eq:large-Y-explicit}
        |\Y_{\ge4}|
        \ge C_2sc^3
\end{equation}
with $C_2=\frac{3\gamma_1\zeta}{4(3+\zeta)^3}>0$.
On the other hand, by \eqref{eq:Sd-crude-cs2} and \eqref{eq:Ble2-less-3s2},
\[
        S_d+|\F_1|+|\F_2|+1\le32cs^2+3s^2+1=o(sc^3),
\]
because $c\ge(\alpha+\varepsilon)s^{2/3}$. Combining this with \eqref{eq:large-Y-explicit} proves \eqref{eq:m3-case3-target} in Subcase~3c.

The ranges $0\le d\le\rho_0 c$, $\rho_0 c<d\le x_1c$, and $x_1c<d\le c-1$ cover all possible values of $d$, since $0<\rho_0<x_1<1$.
This completes Case~3.

\subsection{Equality cases}

The three cases prove \eqref{eq:key-high-terminal}, hence $|\F|\le|\mathcal P'(3,s,\ell)|$.
It remains to discuss equality.

Suppose equality holds.
Cases 1 and 3 yield strict inequalities.
In Case 2, Subcase~2a is strict, and all subcases with $\F_1\cup\F_2\ne\emptyset$ are strict.
Therefore equality can only occur when $\F_1=\F_2=\emptyset$, $|H|=A_3$, and $\Y_{\ge4}=\emptyset$.
Since equality can only occur in Case~2, we have $\nu(H)\le a=\ell-1$.
Since $\nu(H)<\ell$ and $n\ge3\ell-1$, applying \cref{thm:EMC3} to $H$ gives
\[
        |H|\le
        \max\left\{\binom n3-\binom{n-\ell+1}{3},\binom{3\ell-1}{3}\right\}.
\]
The first alternative is $\binom n3-\binom{n-\ell+1}{3}=\binom n3-\binom{r_0}{3}$, which is strictly smaller than $A_3$, because $A_3>\Lambda_3=\binom a2+\binom n3-\binom{r_0}{3}$.
Hence the equality statement in \cref{thm:EMC3} gives $H=\binom{L'}3$ for some $L'\in\binom{[n]}{3\ell-1}$.
Since $\Y_{\ge4}=\emptyset$, we obtain $\F=\binom{L'}3\cup\binom{[n]}{\ge4}=\mathcal P'(3,s,\ell;L')$.

Finally, no extremal family can contain $\emptyset$.
If one did, the initial replacement by a missing singleton would produce an extremal family containing a singleton, whereas the equality family just obtained contains no singleton.
Thus equality holds exactly for the families $\mathcal P'(3,s,\ell;L')$ with $L'\in\binom{[n]}{3\ell-1}$.
This proves \cref{thm:m3-high-range}.

\section{Proof of \cref{thm:conj}}
\label{sec:counterexample}

Fix an integer $m\ge3$.
Throughout this section we write $n=ms+c$ and $\ell=s-c$.
We prove that, for a suitable range of $c$, $\mathcal R(m,s,\ell)$ is larger than all four families appearing in \cref{conj}, and hence \cref{conj} fails.

We first show that $\nu(\mathcal R(m,s,\ell))<s$.
Recall that $\mathcal R(m,s,\ell)\coloneqq\{E\in 2^{[n]}:(m-1)|E|+|E\cap[ms-(m-1)c-1]|\ge m^2\}$.
If $A_1,\ldots,A_s$ are disjoint sets in $\mathcal R(m,s,\ell)$, then a contradiction follows from
\[
        sm^2-1=(m-1)n+\bigl(ms-(m-1)c-1\bigr)\ge\sum_{i=1}^s \bigl((m-1)|A_i|+|A_i\cap[ms-(m-1)c-1]|\bigr)\ge sm^2.
\]

For a family $\mathcal A\subseteq2^{[n]}$, write $\mathcal A^c=2^{[n]}\setminus\mathcal A$.
It is enough to compare the size of their complements.
Let $U_0=[ms-(m-1)c-1]$ and $V_0=[n]\setminus U_0$.
For $E\subseteq[n]$, if $i=|E\cap U_0|$ and $j=|E\cap V_0|$, then $E\notin\mathcal R(m,s,\ell)$ precisely when $mi+(m-1)j\le m^2-1$.
Consequently,
\begin{equation}\label{eq:R-complement-exact}
        |\mathcal R(m,s,\ell)^c|
        =
        \sum_{\substack{i,j\ge0\\ mi+(m-1)j\le m^2-1}}
        \binom{ms-(m-1)c-1}{i}\binom{mc+1}{j}.
\end{equation}
We first derive a uniform upper bound for this complement.
Assume $c\le s^{(m-1)/m}$.
There is a constant $K_0=K_0(m)>0$ such that
\begin{equation}\label{eq:R-complement-upper}
        |\mathcal R(m,s,\ell)^c|
        \le K_0s^{m-1}c
\end{equation}
for all sufficiently large $s$ and all $c\ge1$.
Indeed, every summand in \eqref{eq:R-complement-exact} is at most $O(s^ic^j)$.
If $j=0$, then $i\le m-1$, so $s^ic^j\le s^{m-1}\le s^{m-1}c$.
If $j\ge1$, then it follows from $mi+(m-1)j\le m^2-1$ that $m(m-1-i)\ge(m-1)(j-1)$.
Indeed, subtracting $m-1$ from $mi+(m-1)j\le m^2-1$ gives $mi+(m-1)(j-1)\le m(m-1)$.
Since $c\le s^{(m-1)/m}$, we have $c^{j-1} \le s^{(m-1)(j-1)/m} \le s^{m-1-i}$.
Thus $s^ic^j\le s^{m-1}c$ for every admissible pair $(i,j)$, proving \eqref{eq:R-complement-upper}.

We now compare $\mathcal R(m,s,\ell)$ with the four proposed candidates.

\medskip
\noindent\textbf{1. Comparison with $\mathcal P'(m,s,\ell)$.}
\medskip

Let $U'=[m\ell-1]=[ms-mc-1]$.
Since $U'\subseteq U_0$ and $|U_0\setminus U'|=c$, the complements have the following exact forms:
\[
        \mathcal P'(m,s,\ell)^c
        =\binom{[n]}{\le m-1}
        \cup\{E\in\binom{[n]}m:E\not\subseteq U'\},
\]
whereas
\[
        \mathcal R(m,s,\ell)^c
        =\binom{[n]}{\le m-1}
        \cup\{E\in\binom{[n]}m:E\not\subseteq U_0\}
        \cup\binom{V_0}{m+1}.
\]
Therefore
\begin{equation}\label{eq:Pprime-R-exact-difference}
        \begin{aligned}
                |\mathcal P'(m,s,\ell)^c|-|\mathcal R(m,s,\ell)^c|
                 & =\binom{|U_0|}{m}-\binom{|U'|}{m}-\binom{|V_0|}{m+1}          \\
                 & =\binom{ms-(m-1)c-1}{m}-\binom{ms-mc-1}{m}-\binom{mc+1}{m+1}.
        \end{aligned}
\end{equation}
For all sufficiently large $s$, since $c=o(s)$ in the range considered below, $ms-mc-1\ge ms/2$.
Hence
\[
        \begin{aligned}
                \binom{ms-(m-1)c-1}{m}-\binom{ms-mc-1}{m}
                 & =\sum_{q=0}^{c-1}\binom{ms-mc-1+q}{m-1} \\
                 & \ge c\binom{ms-mc-1}{m-1}
                \ge \mu_1s^{m-1}c
        \end{aligned}
\]
for some constant $\mu_1=\mu_1(m)>0$.
Also $\binom{mc+1}{m+1} \le \mu_2c^{m+1}$ for some constant $\mu_2=\mu_2(m)>0$.
Choose $\beta_{\mathrm R}=\beta_{\mathrm R}(m)>0$ so small that $\beta_{\mathrm R}\le1$ and
\[
        \mu_2\beta_{\mathrm R}^m<\frac{\mu_1}{2}.
\]
If $c\le\beta_{\mathrm R}s^{(m-1)/m}$, then
\[
        \binom{mc+1}{m+1}
        \le \mu_2\beta_{\mathrm R}^m s^{m-1}c
        <\frac{\mu_1}{2}s^{m-1}c.
\]
Together with \eqref{eq:Pprime-R-exact-difference}, this gives $|\mathcal P'(m,s,\ell)^c|>|\mathcal R(m,s,\ell)^c|$, and hence $|\mathcal R(m,s,\ell)|>|\mathcal P'(m,s,\ell)|$.

\medskip
\noindent\textbf{2. Comparison with $\mathcal P(m,s,\ell)$.}
\medskip

Let $L\in\binom{[n]}{\ell-1}$.
A set $E$ is outside $\mathcal P(m,s,\ell;L)$ exactly when $|E|+|E\cap L|\le m$.
In particular, every $m$-set contained in $[n]\setminus L$ belongs to $\mathcal P(m,s,\ell;L)^c$.
Since $|[n]\setminus L|=n-(\ell-1)=(m-1)s+2c+1$, there is a constant $\mu_3=\mu_3(m)>0$ such that
\[
        |\mathcal P(m,s,\ell)^c|
        \ge \binom{(m-1)s+2c+1}{m}
        \ge \mu_3s^m
\]
for all sufficiently large $s$.
On the other hand, by \eqref{eq:R-complement-upper} and $c\le\beta_{\mathrm R}s^{(m-1)/m}$,
\[
        |\mathcal R(m,s,\ell)^c|
        \le K_0\beta_{\mathrm R}s^{m-1/m}=o(s^m).
\]
Thus, for sufficiently large $s$, $|\mathcal R(m,s,\ell)|>|\mathcal P(m,s,\ell)|$.

\medskip
\noindent\textbf{3. Comparison with $\mathcal Q(m,s,\ell)$.}
\medskip

Let $U_1=[ms-c-1]$ and $V_1=[n]\setminus U_1$.
Then $|V_1|=2c+1$.
The complement of $\mathcal Q(m,s,\ell)$ consists of sets satisfying $2|E\cap U_1|+|E\cap V_1|\le2m-1$.
In particular, $\binom{V_1}{2m-1}\subseteq\mathcal Q(m,s,\ell)^c$.
Hence, for some constant $\mu_4=\mu_4(m)>0$,
\begin{equation}\label{eq:Q-complement-lower}
        |\mathcal Q(m,s,\ell)^c|
        \ge \binom{2c+1}{2m-1}
        \ge \mu_4c^{2m-1}
\end{equation}
for all sufficiently large $s$.
Choose $\alpha_{\mathrm R}=\alpha_{\mathrm R}(m)>0$ large enough so that $\mu_4\alpha_{\mathrm R}^{2m-2}>2K_0$.
If $c\ge\alpha_{\mathrm R}s^{1/2}$, then by \eqref{eq:Q-complement-lower},
\[
        |\mathcal Q(m,s,\ell)^c|
        \ge \mu_4c\cdot c^{2m-2}
        \ge \mu_4\alpha_{\mathrm R}^{2m-2}s^{m-1}c
        >2K_0s^{m-1}c.
\]
Using \eqref{eq:R-complement-upper}, we obtain $|\mathcal R(m,s,\ell)|>|\mathcal Q(m,s,\ell)|$.

\medskip
\noindent\textbf{4. Comparison with $\mathcal W(m,s,\ell)$.}
\medskip

The complement of $\mathcal W(m,s,\ell)$ is $\mathcal W(m,s,\ell)^c
        =\bigl\{E\subseteq[n]:|E\cap[ms-1]|\le m-1\bigr\}$.
Since $|[n]\setminus[ms-1]|=c+1$, we have
\[
        |\mathcal W(m,s,\ell)^c|
        =2^{c+1}\sum_{i=0}^{m-1}\binom{ms-1}{i}
        \ge2^{c+1}\binom{ms-1}{m-1}.
\]
Thus, for some constant $\mu_5=\mu_5(m)>0$,
\[
        |\mathcal W(m,s,\ell)^c|
        \ge \mu_5 2^cs^{m-1}.
\]
Since $c\ge\alpha_{\mathrm R}s^{1/2}$, we have $2^c/c\to\infty$.
Therefore, for all sufficiently large $s$,
\[
        \mu_5 2^cs^{m-1}>K_0s^{m-1}c
        \ge |\mathcal R(m,s,\ell)^c|.
\]
Consequently, $|\mathcal R(m,s,\ell)|>|\mathcal W(m,s,\ell)|$.

\medskip

The arguments above show that
\[
        |\mathcal R(m,s,\ell)|
        >
        \max\bigl\{
        |\mathcal P(m,s,\ell)|,
        |\mathcal P'(m,s,\ell)|,
        |\mathcal Q(m,s,\ell)|,
        |\mathcal W(m,s,\ell)|
        \bigr\}
\]
whenever $s$ is sufficiently large and $\alpha_{\mathrm R}s^{1/2}\le c\le\beta_{\mathrm R}s^{(m-1)/m}$.
Since $\mathcal R(m,s,\ell)$ has matching number less than $s$, we obtain
\[
        e(n,s)\ge |\mathcal R(m,s,\ell)|
        >
        \max\bigl\{
        |\mathcal P(m,s,\ell)|,
        |\mathcal P'(m,s,\ell)|,
        |\mathcal Q(m,s,\ell)|,
        |\mathcal W(m,s,\ell)|
        \bigr\}.
\]
This proves \cref{thm:conj}.

\appendix

\section{Technical estimates for \cref{sec:m3-high}}
\label{app:m3-technical-estimates}

\begin{proof}[Proof of \cref{clm:m3-shifted-estimates}]
        For \eqref{eq:m3-shifted-h3}, direct expansion gives
        \begin{equation}\label{eq:shifted-hm-1}
                \Lambda_3-L_1-h_3(n-1,\ell-4)-B_{\le2}
                =\frac{13\ell^2-46\ell s-11\ell+41s^2+17s+4}{2},
        \end{equation}
        and
        \begin{equation}\label{eq:shifted-hm-2}
                \Lambda_3-L_2-h_3(n-2,\ell-3)-B_{\le2}
                =\frac{5\ell^2-14\ell s+5\ell+9s^2-15s+8}{2}.
        \end{equation}
        Since $\ell=s-c$, the numerators in \eqref{eq:shifted-hm-1} and \eqref{eq:shifted-hm-2} are
        \[
                8s^2+20cs+13c^2+6s+11c+4 \text{ and } 4cs+5c^2-10s-5c+8,
        \]
        respectively.
        These two numerators are positive for all sufficiently large $s$, because $c>(\alpha+\varepsilon)s^{2/3}$.
        Therefore, $\Lambda_3-L_i-h_3(N_i,p_i-1)>B_{\le2}$ for $i=1,2$.
        Since $A_3>\Lambda_3$ by \eqref{eq:A3-beats-Lambda3}, \eqref{eq:m3-shifted-h3} follows.

        For \eqref{eq:m3-shifted-clique}, direct expansion gives
        \begin{equation}\label{eq:clique-shift-1}
                6\left(A_3-L_1-\binom{3p_1-1}{3}-B_{\le2}\right)
                =240\ell^2+30\ell s-69s^2-1068\ell+51s+1314,
        \end{equation}
        and
        \begin{equation}\label{eq:clique-shift-2}
                6\left(A_3-L_2-\binom{3p_2-1}{3}-B_{\le2}\right)
                =156\ell^2+54\ell s-117s^2-570\ell+111s+480.
        \end{equation}
        Let
        \[
                \alpha_*=(\sqrt{321}-9)/10.
        \]
        Then $t(s)/s=\alpha_*+O(s^{-1})$. Since $\ell>t(s)$, there is a constant $C>0$ such that $\ell/s\ge \alpha_* - C/s$ for all sufficiently large $s$.
        Define $f_1(x)=240x^2+30x-69$ and $f_2(x)=156x^2+54x-117$.
        Both $f_1$ and $f_2$ are increasing on $x>0$, and
        \[
                f_1(\alpha_*)=\frac{4344-201\sqrt{321}}5>0 \text{ and }
                f_2(\alpha_*)=\frac{11538-567\sqrt{321}}{25}>0.
        \]
        Hence $f_i(\ell/s)>0$ for $i=1,2$ and all sufficiently large $s$.
        The leading quadratic parts of \eqref{eq:clique-shift-1} and \eqref{eq:clique-shift-2} are respectively
        $s^2f_1(\ell/s)$ and $s^2f_2(\ell/s)$, while the remaining terms are $O(s)$.
        Thus both expressions are positive for all sufficiently large $s$, proving \eqref{eq:m3-shifted-clique}.
\end{proof}

\begin{proof}[Proof of \cref{clm:second-term}]
        Since $n=4s-\ell$, we have $n-\ell-d=4s-2\ell-d=2s+2c-d$.
        Set
        \[
                \Delta_d
                \coloneqq
                \binom{3\ell+3d+2}{3}
                -\left(\binom n3-\binom{4s-2\ell-d}{3}\right).
        \]
        It suffices to prove $\Delta_d>0$.
        Let $Q_s(x)=10x^2+(18s-17)x-24s^2+6s+6$.
        By the definition of $t(s)$, it is the larger root of $Q_s(x)$.
        Since $Q_s$ has positive leading coefficient and $\ell>t(s)$, we have $Q_s(\ell)>0$.
        Direct expansion gives
        \[
                \Delta_0=\frac{1}{3}\ell\bigl(Q_s(\ell)+26\ell+6s-4\bigr)>0.
        \]
        For $0\le d\le c-2$, direct calculation gives
        \[
                2(\Delta_{d+1}-\Delta_d)
                =d(26d+50\ell+8s+42)+C_s(\ell),
        \]
        where
        \[
                \begin{aligned}
                        C_s(\ell)
                         & =23\ell^2+16s\ell+39\ell-16s^2+12s+18           \\
                         & =Q_s(\ell)+(13\ell^2-2s\ell+8s^2)+56\ell+6s+12.
                \end{aligned}
        \]
        The quadratic $13\ell^2-2s\ell+8s^2$ is positive for all real $\ell$, since its discriminant is negative.
        Hence $C_s(\ell)>0$, and therefore $\Delta_{d+1}>\Delta_d$ for $0\le d\le c-2$.
        Together with $\Delta_0>0$, this proves $\Delta_d>0$ for every $0\le d\le c-1$.
        Consequently, the second term in \eqref{equ:H-maximum} is larger, so $|H|\le \binom{3\ell+3d+2}{3}$.
        By the definition of $S_d$, this also gives $|H|-A_3\le S_d$.
\end{proof}

\begin{proof}[Proof of \cref{clm:small-shifted-second-term}]
        For $j\in\{1,2\}$, define
        \[
                \Delta_{d,j}
                =
                \binom{3\ell+3d-1}{3}
                -\left(
                \binom{n-j}{3}-\binom{2s+2c-d-j+1}{3}
                \right).
        \]
        Recall that $Q_s(\ell)=10\ell^2+(18s-17)\ell-24s^2+6s+6>0$.
        Direct expansion gives
        \[
                \begin{aligned}
                        6\Delta_{d,1}
                         & =(\ell+d-1)\bigl(2Q_s(\ell)+24s-6\ell-12
                        +d(49\ell+12s-31)+26d^2\bigr),               \\
                        6\Delta_{d,2}
                         & =(\ell+d-1)\bigl(2Q_s(\ell)+48s-15\ell-24
                        +d(49\ell+12s-34)+26d^2\bigr).
                \end{aligned}
        \]
        Since $\ell+d-1>0$ and $Q_s(\ell)>0$, it remains only to check the displayed lower-order terms.
        If $d=0$, then $\ell<s$ gives $24s-6\ell-12>18s-12>0$ and $48s-15\ell-24>33s-24>0$ for all sufficiently large $s$.
        If $d\ge1$, then $49\ell+12s-31$ and $49\ell+12s-34$ are positive for all sufficiently large $s$.
        Therefore $\Delta_{d,j}>0$ for $j=1,2$, proving the claim.
\end{proof}

\end{document}